\documentclass[12pt]{article}
\usepackage{graphicx} % Required for inserting images
\usepackage{a4wide}
\usepackage{fancyhdr}

\fancypagestyle{plain}{% Redefine the plain style
  \fancyhf{} % Clear all headers and footers
  \fancyfoot[C]{\thepage} % Add page number to the center of the footer
}

\usepackage[T1]{fontenc}
\usepackage{amssymb}
\usepackage{hyperref}
\usepackage{epstopdf}
\usepackage{makeidx}
\usepackage{amsfonts,amsmath}
\usepackage{amsthm}
\usepackage[usenames,dvipsnames]{xcolor}
\usepackage{palatino}
\usepackage{extarrows}
\usepackage{mathtools}
\usepackage{comment}
\usepackage{tocloft}
\usepackage{enumitem}
\usepackage{soul}

\newtheorem{theo}{Theorem}[section]
\newtheorem{remark}{Remark}
\newtheorem{prop}{Proposition}
\newtheorem{lemma}{Lemma}
\newtheorem{definition}{Definition}
\newtheorem{corol}{Corollary}

\usepackage[maxbibnames=99]{biblatex}
\renewbibmacro{in:}{}
\renewbibmacro*{issue+date}{%
  \setunit{\addcomma\space}% NEW
    \iffieldundef{issue}
      {\usebibmacro{date}}
      {\printfield{issue}%
       \setunit*{\addspace}%
       \usebibmacro{date}}% NEW
      \newunit}

\newcommand{\de}{{^\dagger}}

\title{Higher order error estimates for regularization of inverse problems under non-additive noise}
\author{Diana-Elena Mirciu$^1$ and Elena Resmerita$^2$\\  
\footnotesize Department of Mathematics, University of Klagenfurt, Austria}
\date{ \scriptsize $^1$E-mail: diana-elena.mirciu@aau.at\\[1ex]
        $^2$E-mail: elena.resmerita@aau.at\\[2ex]
        \normalsize \today}

\begin{document}

\maketitle

\begin{abstract}
    In this work we derive higher order error estimates for inverse problems distorted by non-additive noise, in terms of Bregman distances. The results are obtained by means of a novel source condition, inspired by the dual problem. Specifically, we focus on variational regularization having the Kullback-Leibler divergence as data-fidelity, and a convex penalty term. In this framework, we provide an interpretation of the new source condition, and present error estimates also when a variational formulation of the source condition is employed. We show that this approach  can be extended to variational regularization that incorporates  more general convex data fidelities.
\end{abstract}

\section{Introduction}
    A prototype of an inverse problem is the operator equation 
        \begin{equation} \label{IP}
        Au=f,
    \end{equation}
    where $X,Y$ are Banach spaces, $A:X\rightarrow Y$ is a bounded linear operator and $f\in Y$. In practice, the exact data $f$ are not known precisely, being instead approximated by $f^{\delta}$ with $\Vert f^{\delta}-f\Vert\leq\delta$.
    If problem \eqref{IP} is ill-posed, one needs to employ regularization methods in order to obtain stable approximations of the solution.

    Throughout this work, we denote by $u\de\in X$ an $R$-minimising solution of  equation \eqref{IP}, i.e.
    \begin{equation}\label{sol}
        u\de\in\operatorname{arg min}\{R(u) \vert Au=f\},
    \end{equation}
    for some function $R$.
    Several approaches can be employed to obtain an approximation of $u\de$, such as iterative methods or variational regularization. The latter amounts to finding a solution for the following optimization problem
    \begin{equation}\label{optim}
        \min\limits_{u\in X} \left\{\frac{1}{\alpha} H_{f}(Au)+R(u)\right\},
    \end{equation}
    where $\alpha>0$ is the regularization parameter which balances the effect of the data fidelity function $H_{f}$ and the penalty term $R$.  We assume that $H_{f}:Y\rightarrow\mathbb{R}\cup\{+\infty\}$ and $R:X\rightarrow\mathbb{R}\cup\{+\infty\}$ are proper, convex and lower semicontinuous functionals.   While the regularizer $R$ is considered in  ways that promote certain features of the solution, such as sparsity, non-negativity, piecewise constant or linear structure, and so on,  $H_{f}$ is usually chosen according to the type of noise in the data.
     We provide further details on this matter.  For additional information, refer to \cite{Benning}, \cite{benning2018modern} and \cite{sixou2018morozov}. Assume that the distribution of the noise is known. The variational approach is linked to a Bayesian estimation, more specifically, the maximum a posteriori probability (MAP) estimators. This amounts to maximizing the posterior probability density $P(u\vert f)$, which is equivalent to minimizing its negative logarithm, namely $-\operatorname{ln }P(u\vert f).$
    For simplicity, consider the situation when the spaces $X,Y$ are finite dimensional. In a discrete setting, Bayes' formula states that the posterior probability distribution $P(u\vert f)$ of measuring $u$ given data $f$ satisfies the following equality, $$P(u\vert f)=\frac{P(f\vert u)P(u)}{P(f)}.$$
    Since $P(f)$ does not depend on $u$, the minimization problem becomes $$\min_{u}\{-\operatorname{ln }P(f\vert u)-\operatorname{ln }P(u)\}.$$ Frequently, one has the a priori distribution modeled by a Gibbs-prior, $$P(u)\sim\operatorname{exp} (-\alpha R(u)),$$ and the probability $P(f\vert u)$ of measuring $f$ given $u$, 
    \begin{equation}\label{probab}
            P(f\vert u)\sim\operatorname{exp }(-H_f(Au))).
    \end{equation}
    When this is substituted in the previous minimization problem, it becomes
    \begin{equation*}
        \min_{u}\{ H_f(Au)+\alpha R(u) \}.
    \end{equation*}
    By comparing \eqref{probab} to the latter, one can notice the connection between the type of noise in the data and the choice of the data fidelity function $H_f$.
    
    A common setting which deals with additive Gaussian noise in the data is the least squares one,
    \begin{equation*}
    H_{f}(z)=\frac{1}{2}\Vert z-f\Vert^2.
    \end{equation*}
    
     However, in applications such as Positron Emission Tomography, the data are corrupted by Poisson noise, for which the quadratic  $H_f$ is not suitable. The standard choice for the data fidelity term in this context is the Kullback-Leibler functional. 
     The variational regularization involving the latter function as a data fitting and/or as a regularizer has been mentioned as a useful approach for numerous problems arising in physics, astronomical and medical imaging - see \cite{sixou2018morozov} on spectral computerized tomography and Section 5.3 in \cite{Engl00} on maximum entropy regularization, and the references therein.

   As a general regularization strategy, one is interested to show the existence of solutions for problem \eqref{optim}, to prove  convergence of these minimizers to the $R$-minimizing solution of \eqref{IP} with respect to some topology, and to establish error estimates.
    
    We will focus here on error estimates. They can be derived under additional assumptions on the solution, called source conditions, which entail regularity properties of the solution, e.g. some Sobolev smoothness (see \cite[Section 10.5]{Engl00}, \cite{hohage2017characterizations}, \cite{hohage2015verification}, \cite{weidling2015variational}).
    
    Let us recall some background on error estimates for inverse problems. 
    See, for example, the book \cite{Engl00}, Chapter 3.2, where error estimates with respect to the norm are presented in the setting of both quadratic data fidelity and penalty term in Hilbert spaces, under the source condition $u^\dagger\in \text{Ran}(A^*A)^\nu$ for $\nu\in (0,1]$.  Here, $\text{Ran}$ stands for the range of an operator and $A^*$ denotes the adjoint of $A$.

    In order to extend such results to the case when $R$ is not necessarily quadratic and $X$ is a Banach space, the work \cite{Eggermont93} introduced a source condition which allowed deriving error estimates for the (non-quadratic) entropy regularization that approximates least squares positive solutions.

    Recall that, for a given $v\in X$, the Bregman distance with respect to $R$ and $\xi\in\partial R(v)$ is the function $D_R(\cdot,v):\operatorname{ dom }R\rightarrow[0,+\infty]$ defined by
    \begin{equation*}
         D_R^{\xi}(u,v)=R(u)-R(v)-\langle\xi,u-v\rangle.
     \end{equation*}
     If $\partial R(v)$ is a singleton, we will simply write $D_R(u,v)$.
     More general, if the directional derivative $R^{\circ}(v,u-v)$ of $R$ at $v$ in the direction $u-v$ exists, one defines (see, e.g. the survey \cite{but-res})
     \begin{equation*}
         D_R(u,v)=R(u)-R(v)-R^{\circ}(v,u-v).
     \end{equation*}
     In addition, for $u,v\in\operatorname{ dom }R$ and $\zeta\in\partial R(u), \xi\in\partial R(v)$, one can work with the symmetric Bregman distance between $u,v$, denoted  generically by $D_R^s $, 
     \begin{equation*}
         D_R^{s}(u,v)=D_R^\xi(u,v)+D_R^\zeta(v,u).
     \end{equation*}
        In \cite{Burger04}, the authors extended the source condition from \cite{Eggermont93} to a more general convex regularization functional $R$, that is 
    $\text{Ran}(A^*)\cap\partial R(u\de)\neq\emptyset,$
    and obtained the following estimates in the Bregman distance with respect to $R$,
\begin{equation}\label{slow}
    D_R(u_{\alpha},u\de)=O(\alpha) \text{ and } D_R(u_{\alpha}^{\delta},u\de)=O(\alpha)+\frac{\delta^2}{2\alpha},
    \end{equation}
    where $u_{\alpha}$ and $u_{\alpha}^{\delta}$ denote the solutions of \eqref{optim} in the situation of exact data $f$ and noisy data $f^\delta$, respectively, while $\delta$ is the noise level. The reader is further referred to  \cite{resmerita-scherzer}, \cite{Lorenz-sparse}, \cite{reich2023}, \cite{grasmair-nonconvex}, \cite{Obmann+haltmeier2024}, \cite{ebner-haltmeier} on error estimates for regularization with specific convex penalties or with non-convex regularizers for (non)linear operator equations and variational problems, as well as for regularization in a data-driven framework.   See also the books \cite{Scherzer08} and \cite{bangti-jin-book}  for a broader view on the topic.
    
    The stronger source condition $\text{Ran}(A^*A)\cap\partial R(u\de)\neq\emptyset$ is dealt with in \cite{Resmerita05}, leading to  improved error estimates as follows,
    \begin{equation}\label{fast}
D_R(u_{\alpha},u\de)\leq D_R(u\de-\alpha v, u\de)\text{ and } D_R(u_{\alpha}^{\delta},u\de)\leq D_R(u\de-\alpha v,u\de) +\frac{\delta^2}{2\alpha},
\end{equation}
    where $v\in X$ is a source element, that is $A^*Av\in\partial R(u^\dagger)$. Under additional assumptions on  $R$, one obtains the convergence rates $O(\alpha^2)$ and $O(\delta^{\frac{4}{3}})$ if $\alpha\sim\delta^{\frac{2}{3}}$, for exact and noisy data, respectively.

     In \cite{Benning}, the author employed the strong source condition $\text{Ran}(A^*F'(A\cdot))\cap\partial R(u\de)\neq\emptyset$ in the more general setting when $H_{f}(z)=F(z-f)$ for a given function $F:Y\rightarrow\mathbb{R}$, which leads to a similar estimate as in \cite{Resmerita05}. Interestingly, this regularity condition highlights the interplay between the data fidelity and the regularizer, which is not noticeable when working with least squares solutions.
     
    An alternative formulation for a source condition can be found in \cite{Hofmann07}, specifically the so-called variational source condition - see also \cite{Bot09}, \cite{Flemming18}, \cite{Scherzer08} and \cite{Hohage19}. This approach is quite suitable for lower order error estimates, as well as for nonlinear operator equations. The reader is referred to the works \cite{hohage2017characterizations}, \cite{hohage2015verification} and \cite{weidling2015variational}, showing the meaning of such conditions for some inverse problems. 
    
     Regarding higher order estimates for \eqref{optim}, that is,  covering also intermediary orders between \eqref{slow} and \eqref{fast}, the article  \cite{Grasmair13} shows that they can be obtained by imposing a (variational) source condition on the solution of the Fenchel dual problem of \eqref{optim}, when $H_f$  is some power of the norm; see also  \cite{Hohage19} for the case when the data fidelity is a more general function of the difference of the arguments, which still corresponds to some kind of additive noise. However, to the best of our knowledge, higher order error estimates for problems distorted by non-additive noise are still missing in the literature. 

    Note that \cite{Benning11} derived (lower order) error estimates in Bregman distance under the weaker assumption $\text{Ran}(A^*)\cap\partial R(u\de)\neq\emptyset$ in the framework of general convex fidelity functions, including the ones appropriate for noise following Poisson or gamma distributions. Moreover, the recent work \cite{benning2024}  advances the quantitative analysis by revealing that source condition elements can be computed as solutions of proximal operator based problems.
    
    Our goal is to obtain higher order error estimates when the data are corrupted by non-additive noise. % and the regularization parameter is chosen according to some a priori rule. 
    To this end, we introduce a novel equality-type source condition for the Fenchel dual of \eqref{optim}, which matches the range of the operator $A$ and the data fidelity. At the same time, we pursue the study under an appropriate variational source condition for the dual problem. We manage to link the two forms of regularity as done for the additive noise case in \cite{Grasmair13, Hohage19}, and provide an interesting interpretation of the new condition, inspired by \cite{Benning11}. For the sake of clarity,  we detail in this work the analysis for the Poisson noise setting, and point out how the techniques can be applied when more general convex data fidelities are involved.

    We provide more insight on the analysis when passing from the additive noise case $F(f-z)$ to the non-additive one. In the first setting, the dual problem contains a linear term, thus simplifying the calculations, whereas such a term is missing in the non-additive  noise case and therefore, is enforced  using a Taylor expansion of an appropriate function. Additionally, we employ a strong convexity property for the Fenchel conjugate of the data fidelity (compare to \cite{Grasmair13}).

        This work is presented as follows.  Assumptions and results useful for the subsequent study are stated in  Section \ref{preliminaries}. 
        Section \ref{exactdatacase} focuses on important relations employing Bregman distances that link primal and dual variables. In Section \ref{fidelitiesscaling} we derive error estimates in the case of data  fidelities with scaling properties, corresponding to  additive noise (compare to \cite{Benning}). Thereafter, in Section \ref{KLfidelity}, we investigate the framework when the data fidelity is the Kullback-Leibler functional,  using an operator range source condition, as well as a variational one. Moreover, we specialize the results for the joint Kullback-Leibler regularization, and   provide an interpretation of the novel source condition.
        %The focus of Section \ref{KLfidelitynoisy} is adapting the main result from Subsection \ref{KLfidelityexact} to the context of noisy data. 
       % Afterwards, we consider the particular instance when the penalty term is the Kullback-Leibler divergence as well, thus promoting non-negative solutions. 
        %Subsection \ref{KLfidelityinterpretation} provides an interpretation of this novel source condition, namely, we prove its equivalence with the fact that $u\de$ is a minimizer of \eqref{optim} with the data $f$ replaced by an element in the range of $A$. In Subsection \ref{KL_VSC} we also establish error estimates for the situation of a variational formulation for the source condition and provide a connection between that assumption and the aforementioned one.
         In Section \ref{generalfidelities} we formulate a generalization of the strong source condition which facilitates attaining error estimates for variational regularization involving a broader class of convex data fidelities. 

\section{Preliminaries}\label{preliminaries}
    Let us  recall some basic results  and introduce some assumptions that will be used in the next sections.
    
    As mentioned in the Introduction, suppose that $H_{f}:Y\rightarrow\mathbb{R}\cup\{+\infty\}$ and $R:X\rightarrow\mathbb{R}\cup\{+\infty\}$ are proper, convex and lower semicontinuous functionals on the corresponding Banach spaces. 
    Throughout this work, we assume that  problem \eqref{optim} admits a minimizer $u_{\alpha}$ for any $\alpha>0$, and  there exists
    $u_0\in\operatorname{dom }(H_f\circ A)\cap\operatorname{dom } R$, such that $H_f$ is continuous at $Au_0$. 
    
        Moreover, let the source condition $\operatorname{Ran }(A^*)\cap\partial R(u\de)\neq\emptyset$ hold for $u\de$ as in \eqref{sol},
    i.e.,
    \begin{equation} \label{SC1}
        \exists\; p\de\in Y^* \text{ such that } \xi\de:=A^*p\de\in\partial R(u\de).
    \end{equation}
    
    By employing Theorem 4.1 and Proposition 4.1 from \cite{Ekeland99}, Part I, Chapter III, we get that the dual problem
    \begin{equation}\label{dualpb}
        -\min\limits_{p\in Y^*}\left\{\frac{1}{\alpha} H^*_f(-\alpha p)+R^* (A^* p)\right\}
    \end{equation}
    admits a solution $p_{\alpha}\in Y^*$, and the following optimality conditions hold: 
    \begin{gather*}
        A^*p_{\alpha}\in\partial R(u_{\alpha}),\\
        -\alpha p_{\alpha}\in\partial H_f(Au_\alpha).
    \end{gather*}

         Let $\Omega,\Gamma\subset\mathbf{R}^2$ be two bounded and measurable sets. Recall that the Boltzmann-Shannon entropy is the function $g:L^1(\Gamma)\rightarrow(-\infty,+\infty]$ given by
    \begin{gather*}
        g(x)=
        \begin{cases}
            \int\limits_{\Omega}{x(t)\text{ ln } x(t)dt}, & \text{ if } x\geq0\text{ a.e. and } x\text{ ln }x\in L^1(\Omega),\\
            +\infty, & \text{ otherwise}
        \end{cases}
    \end{gather*}
    while the Kullback-Leibler functional is defined as the Bregman distance with respect to the Boltzmann-Shannon entropy, that is $$KL:L^1(\Gamma)\times L^1(\Gamma)\rightarrow[0,+\infty], \; KL(u,v)=D_g(u,v).$$ 
    Therefore, one has (see, e.g., \cite{Resmerita05}),
    \begin{gather} \label{KLdef}
        KL(u,v)=
            \begin{cases}
                \int\limits_{\Gamma}\left[u(t)\text{ ln }\frac{u(t)}{v(t)}-u(t)+v(t) \right]dt, & \text{ if } u,v\in\operatorname{ dom }g,\\
                +\infty, & \text{ otherwise.}
            \end{cases}
    \end{gather}

        It is known that the positive orthant of any $L^q$ space has an empty interior when $q\in(1,\infty)$, as opposed to the case $q=\infty$ - see \cite{Borwein-Limber}. Therefore, as in  Section 4.1.3 in \cite{jung} and the references therein, we consider the convex set 
    \begin{equation}\label{defE1}
        E=\{v\in L^1(\Gamma) \vert 0<\operatorname{essinf} v\leq\operatorname{esssup} v<+\infty\}\subset L^{\infty}(\Gamma).
    \end{equation}
    This will play an important role in our analysis related to the Kullback-Leibler fidelity, due to the property 
    \begin{equation}\label{E1int}
        E=\operatorname{int }_{L^{\infty}(\Gamma)}L^{\infty}_+(\Gamma).
    \end{equation}
    Recall that the essential infimum and essential supremum of a function $v$ are given by 
    \begin{gather*}
        \operatorname{ess inf} v=\operatorname{sup}\{\lambda\in\mathbb{R} \cup\{-\infty, +\infty\} : v\geq\lambda \text{ a.e.}\},\\
        \operatorname{ess sup} v=\operatorname{inf}\{\lambda\in\mathbb{R} \cup\{-\infty, +\infty\} : v\leq\lambda \text{ a.e.}\}.
    \end{gather*}

        The following statement will be instrumental in the forthcoming sections.
    \begin{theo}[{\cite[356]{Attouch14}}][Theorem 9.5.1]\label{lemabenninggen}
         Let $G:X\rightarrow\mathbb{R}\cup\{+\infty\}$ be a proper, convex and lower semicontinuous function. Then, for any $u\in X$ and $v\in X^*$, one has $$v\in\partial G(u) \Leftrightarrow u\in\partial G^*(v).$$
     \end{theo}
     
     The next lemma is a particular instance of the previous theorem.
     \begin{lemma}[{\cite[51]{Benning}}][Lemma 2.7]\label{lemabenning}
         Let $G:X\rightarrow\mathbb{R}$ be a Fr\'echet differentiable functional with invertible Fr\'echet derivative $G'$, i.e. $(G')^{-1}(v)$ exists for all $v\in X^*$. Then the inverse of the Fr\'echet derivative is the derivative of the convex conjugate of $G$, i.e., $$\forall\; v\in X^*,\quad (G^*)'(v)=(G')^{-1}(v).$$
     \end{lemma}
     
     By employing Theorem \ref{lemabenninggen}, one gets equivalent formulations of the optimality conditions,
     \begin{gather}
        \xi_{\alpha}:=A^*p_{\alpha}\in\partial R(u_{\alpha}), \label{OC1}\\
        \zeta_{\alpha}:=Au_\alpha\in\partial H_f^*(-\alpha p_{\alpha}). \label{OC2}
    \end{gather}
Moreover, assume that $\partial H_f^*(-\alpha p\de)\neq\emptyset$, and let us consider a subgradient $\zeta^{\dagger}_{\alpha}\in\partial H_f^*(-\alpha p\de)$ which is needed for dealing with certain Bregman distances.    
\section{ Error estimates for general convex fidelities} \label{exactdatacase}
Firstly, we focus on establishing error estimates for \eqref{optim} in  the exact data situation. With this purpose in mind, we begin by stating a couple of essential relations employing the Bregman distances associated to functions that are relevant for the primal and the dual problems.
\begin{lemma}
If   source condition \eqref{SC1} holds, then the following statement  is true,
\begin{equation}\label{egalitateBregman}
    D_R^s(u_\alpha,u\de)+\frac{1}{\alpha} D_{H^*_f}^s(-\alpha p_\alpha,-\alpha p\de) = \langle f,p\de-p_{\alpha}\rangle + \langle \zeta^{\dagger}_{\alpha},p_{\alpha}-p\de\rangle,
\end{equation}
where $D_{H^*_f}^s(-\alpha p_\alpha,-\alpha p\de)$ is the symmetric Bregman distance between $-\alpha p_\alpha,-\alpha p\de$ with respect to the subgradients $\zeta_{\alpha},\zeta^{\dagger}_{\alpha}$, respectively, while $D_R^s(u_\alpha,u\de)$ is defined via $\xi^\dagger$ and  $\xi_\alpha$ given by \eqref{SC1} and \eqref{OC1}.
\end{lemma}
\begin{proof}
    We start by evaluating 
\begin{align*}
    D_R^s(u_\alpha,u\de)&=D_R(u_\alpha,u\de)+D_R(u\de,u_\alpha)
    \\&=R(u_\alpha)-R(u\de)-\langle u_\alpha-u\de,A^*p\de\rangle+R(u\de)-R(u_\alpha)-\langle u\de-u_{\alpha},A^*p_{\alpha}\rangle
    \\&=\langle u\de-u_{\alpha},A^*p\de-A^*p_{\alpha}\rangle=\langle f,p\de-p_{\alpha}\rangle +\langle Au_{\alpha},p_{\alpha}-p\de\rangle,
\end{align*}
where  \eqref{OC1} and  \eqref{SC1} were used.
Due to \eqref{OC2}, we get 
\begin{align*}
    D_R^s(u_\alpha,u\de)&=\langle f,p\de-p_{\alpha}\rangle +\langle \zeta_{\alpha},p_{\alpha}-p\de\rangle\\
    &= \langle f,p\de-p_{\alpha}\rangle -\frac{1}{\alpha}\langle \zeta_{\alpha}-\zeta^{\dagger}_{\alpha},-\alpha p_{\alpha}+\alpha p\de\rangle+\langle \zeta^{\dagger}_{\alpha},p_{\alpha}-p\de\rangle,
\end{align*}
which can be rewritten as
\begin{equation*}
    D_R^s(u_\alpha,u\de)+\frac{1}{\alpha} D_{H^*_f}^s(-\alpha p_\alpha,-\alpha p\de) = \langle f,p\de-p_{\alpha}\rangle + \langle \zeta^{\dagger}_{\alpha},p_{\alpha}-p\de\rangle.
\end{equation*}
\end{proof}

The following theorem provides a way to derive error estimates, under the assumption that a certain inequality involving the subgradient of the Fenchel conjugate of the data fidelity holds.
\begin{theo}\label{bregmanestim}
    If  source condition \eqref{SC1} holds and there exists $v\in X$ and $s:[0,+\infty)\rightarrow\mathbb{R}$ such that $s(\alpha)\rightarrow0, $ for $\alpha\rightarrow0$ and
    \begin{equation} \label{egalit}
        \langle\zeta^{\dagger}_{\alpha}, p_{\alpha} -p\de\rangle - \frac{1}{\alpha} D_{H^*_f}^s(-\alpha p_\alpha,-\alpha p\de) \leq \langle f,p_{\alpha}-p\de\rangle-s(\alpha)\langle Av, p_{\alpha}-p\de\rangle+ O(\alpha^r),
    \end{equation}
     is verified for some $r\in\mathbb{N}$, then  the following inequality is true: 
    \begin{equation*}
        D_{R}(u_{\alpha}, u\de)\leq D_{R}(u\de-s(\alpha)v, u\de) + O(\alpha^r).
    \end{equation*}
    Moreover, assuming that \eqref{egalit}  holds without the term $O(\alpha^r)$, then the conclusion remains valid without this term, i.e.
    \begin{equation*}
        D_{R}(u_{\alpha}, u\de)\leq D_{R}(u\de-s(\alpha)v, u\de).
    \end{equation*}
\end{theo}
\begin{proof} 
If \eqref{egalit} holds, from \eqref{egalitateBregman} one obtains
\begin{equation*}
    D_R^s(u_\alpha,u\de) \leq -s(\alpha)\langle Av, p_{\alpha}-p\de\rangle+ O(\alpha^r),
\end{equation*} 
and from the definition of the symmetric Bregman distance and \eqref{OC1} we have
\begin{align*}
    D_R(u_\alpha,u\de) &\leq -D_R(u\de,u_{\alpha}) - s(\alpha)\langle Av, p_{\alpha}-p\de\rangle+ O(\alpha^r)\\
    &= R(u_{\alpha})-R(u\de)+\langle A^*p_{\alpha},u\de- s(\alpha)v-u_{\alpha}\rangle + s(\alpha)\langle Av,p\de\rangle+ O(\alpha^r).
\end{align*}
Condition \eqref{OC1} implies that $\langle A^*p_{\alpha},u\de- s(\alpha)v-u_{\alpha}\rangle\leq R(u\de- s(\alpha)v)-R(u_{\alpha})$ which, together with the previous inequality, gives
\begin{align*}
    D_R(u_\alpha,u\de) &\leq R(u\de- s(\alpha)v)-R(u\de)+s(\alpha)\langle v,A^*p\de\rangle+ O(\alpha^r)\\
    &=D_{R}(u\de-s(\alpha)v, u\de) + O(\alpha^r).
\end{align*}
If \eqref{egalit} is true without the term $O(\alpha^r)$, one can arrive 
in  a similar way to the corresponding inequality.% $$D_{R}(u_{\alpha}, u\de)\leq D_{R}(u\de-s(\alpha)v, u\de).$$
\end{proof}

\begin{remark}\label{remarkord}
    For example, if $R$ is assumed to be continuous, then $\lim\limits_{\alpha\rightarrow0}D_R(u\de-s(\alpha)v,u\de)=0$ holds. Moreover, if $R$ is twice differentiable around $u\de$ and the second derivative is bounded around $u\de$, then one obtains  $D_R(u\de-s(\alpha)v,u\de)=O(s(\alpha)^2)$. 
\end{remark}

    Assume now that the exact data $f$ are not known, but we only have access to an approximation $f^{\delta}$ satisfying
    \begin{equation*}
        \Vert f-f^{\delta}\Vert\leq\delta,
    \end{equation*}
    with respect to some norm.
    We remark that a more natural assumption would be $H_{f^{\delta}}(f)\leq h(\delta)$ for some appropriate function $h$, but for technical reasons we assume the norm inequality.     
    Note that this stratagem has already been employed in the Poisson noise literature - see, e.g., \cite{haltmeier2009regularization}.
    
    Let $u_{\alpha}^{\delta}$ and $p_{\alpha}^{\delta}$ be solutions for the primal and the dual problem, respectively. They must satisfy the conditions of the Fenchel-Rockafellar Theorem,
    \begin{gather}
        \xi_{\alpha}^{\delta}:=A^*p_{\alpha}^{\delta}\in\partial R(u_{\alpha}^{\delta}) \label{OC4},\\
        \zeta_{\alpha}^{\delta}:=Au_\alpha^{\delta}\in\partial H_{f^{\delta}}^*(-\alpha p_{\alpha}^{\delta}). \nonumber
    \end{gather}
    Assume that $\partial H_{f^{\delta}}^*(-\alpha p\de)\neq\emptyset$ and consider
    \begin{equation}\label{zeta}
        \zeta_{\alpha}^{\dagger}\in\partial H_{f^{\delta}}^*(-\alpha p\de).
    \end{equation}
    Theorem \ref{bregmanestim} can be formulated for the noisy data case as follows.
    \begin{theo}
    If  source condition \eqref{SC1} holds, there exists $C>0$, $v\in X$ and $s:[0,+\infty)\rightarrow\mathbb{R}$ such that $s(\alpha)\rightarrow0, $ for $\alpha\rightarrow0$ and
    \begin{equation} \label{egalit}
        \langle\zeta^{\dagger}_{\alpha}, p_{\alpha}^{\delta} -p\de\rangle - \frac{1}{\alpha} D_{H^*_{f^{\delta}}}^s(-\alpha p_\alpha^{\delta},-\alpha p\de) \leq \langle f,p_{\alpha}^{\delta}-p\de\rangle-s(\alpha)\langle Av, p_{\alpha}^{\delta}-p\de\rangle+ C\frac{\delta^2}{\alpha}+O(\alpha^r),
    \end{equation}
     is verified for some $r\in\mathbb{N}$, then  the following inequality is true: 
      \begin{equation*}
          D_{R}(u_{\alpha}^{\delta}, u\de)\leq D_{R}(u\de-s(\alpha)v, u\de) + C\frac{\delta^2}{\alpha} + O(\alpha^r).
        \end{equation*}
    \end{theo}
    \begin{proof}
        The proof is similar to the one for the exact data setting.
    \end{proof}

In the following sections, we show that \eqref{egalit} holds for several important instances, that is, for the quite popular data fidelities  $H_f(Au)=F(Au-f)$, the Kullback-Leibler  distance $H_f(Au)=KL(f,Au)$,  and even for more general data-fitting terms. We will provide all the details for the $KL$ setting, including the noisy data case.

\section{Fidelities with scaling properties} \label{fidelitiesscaling}
    From now on, all function equalities will be understood pointwise. 
    For a fixed $f\in Y$, consider the setting $H_f(Au)=F(Au-f)$, where $F:Y\rightarrow\mathbb{R}$ is convex and Fr\'echet differentiable. 
    Assume that the source condition from \cite{Benning} holds, namely 
    \begin{equation}\label{SC2}
    \exists\; p\de\in Y^*, v\de\in X \text{ such that } A^*p\de\in\partial R(u\de) \text{ and }p\de=-F'(-Av\de).
        %\operatorname{Ran }(-A^*(F'(A\cdot)))\cap\partial R(u\de)\neq\emptyset,
    \end{equation}
    Clearly, if $(F')^{-1}:Y^*\rightarrow Y$ exists, the latter is equivalent to  $Av\de=-(F^*)'(-p\de)$, that is,   $(F^*)'(-p\de)\in \text{Ran}(A)$.
    
     The next result extends Theorem 5.3 from \cite{Benning}, which was proven under the additional assumption that the functional $R$ is absolutely one-homogeneous (i.e. for any $c\in\mathbb{R}$ and $u\in X$, $R(cu)=\vert c\vert R(u)$), to the more general case of a proper, convex and lower semicontinuous function. Note that \eqref{dualpb} reads now as follows (see details below), 
    \begin{equation*}
        -\min\limits_{p\in Y^*}\left\{\frac{1}{\alpha} F^*(-\alpha p)-\langle p,f\rangle+R^* (A^* p)\right\}.
    \end{equation*}
    
    \begin{theo}\label{thscaling}
        Assume that $F:Y\rightarrow\mathbb{R}$ is Fr\'echet differentiable, $(F')^{-1}:Y^*\rightarrow Y$ exists, source condition \eqref{SC2} holds and $F$ satisfies the scaling property
        \begin{equation}\label{scalingprop}
        \exists\; s:[0,+\infty)\rightarrow\mathbb{R} \text{ such that } \lim\limits_{\alpha\rightarrow0}s(\alpha)=0 \text{ and }\forall\;\lambda>0, \forall w\in Y: \lambda F'(w)=F'(s(\lambda)w).
        \end{equation} 
        Then the following inequality is true,
        \begin{equation*}
            D_R(u_\alpha,u\de)\leq D_{R}(u\de-s(\alpha)v\de, u\de).
        \end{equation*}
    \end{theo}
    \begin{proof}
        We will show that \eqref{egalit} is satisfied without the $O(\alpha^r)$ term and then we get the conclusion by employing Theorem \ref{bregmanestim}. 
        To this end, one can use formula (4.9), Part I, Chapter I from \cite{Ekeland99} to find the Fenchel conjugate of $H_f(z)=F(z-f)$, $$\forall z^*\in Y^*, \quad H^*_f(z^*)=F^*(z^*)+\langle z^*,f\rangle.$$ 
        Lemma \ref{lemabenning} implies the differentiability of $F^*$, which leads to $(H^*_f)'(z^*)=(F^*)'(z^*)+f.$
        From \eqref{scalingprop} it follows that,   for any $w^*\in Y^*$,
\begin{equation}\label{scalingprop2}
            (F^*)'(\alpha w^*)=(F^*)'(\alpha F'((F^*)'(w^*))) = (F^*)'(F'(s(\alpha)(F^*)'(w^*))) = s(\alpha)(F^*)'(w^*).
        \end{equation}
        We evaluate the left hand side of \eqref{egalit} in this particular case. Since 
        $D_{H^*_f}^s(-\alpha p_\alpha,-\alpha p\de)\geq0$, one gets due to \eqref{scalingprop2},
       
        \begin{align*}
            \langle(H^*_f)'(-\alpha p\de), p_{\alpha} -p\de\rangle - \frac{1}{\alpha} D_{H^*_f}^s(-\alpha p_\alpha,-\alpha p\de) &\leq \langle f,p_{\alpha} -p\de\rangle+\langle(F^*)'(-\alpha p\de), p_{\alpha} -p\de\rangle \\ &=\langle f,p_{\alpha} -p\de\rangle+s(\alpha)\langle(F^*)'(-p\de), p_{\alpha} -p\de\rangle\\
            &= \langle f,p_{\alpha} -p\de\rangle-s(\alpha)\langle Av\de, p_{\alpha} -p\de\rangle,
        \end{align*}
        where for the last equality we employed the equivalent formulation of  source condition \eqref{SC2}. Finally, this shows that \eqref{egalit} holds without the $O(\alpha^r)$ term, hence the conclusion follows from Theorem \ref{bregmanestim}.
        \begin{remark}   
              If $F(z)=\frac{1}{q}\Vert z\Vert^q$ for some $q>1$ and $Y$ is smooth, strictly convex and reflexive (for more details see \cite{Grasmair13}), then it is known that the of $F$ is given by the duality mapping
              \begin{equation*}
                  J_q(z)=\{z^*\in Y^* : z^*(z)=\Vert z^*\Vert\cdot\Vert z\Vert, \Vert z^*\Vert=\Vert z\Vert^{q-1}\}.
              \end{equation*}
              The smoothness of $Y$ implies that $J_q$ is single valued, according to \cite{cioranescu2012geometry}, Chapter I, Corollary 4.5. 
                Using the assumptions that $Y$ is reflexive and strictly convex, and Corollary 1.4 from Chapter II in \cite{cioranescu2012geometry},  it follows that $Y^*$ is smooth. This, together with the previous result gives $J_{q^*}$ single valued.
              \\
               The mapping $J_q$ is $(q-1)$-homogeneous (and single-valued in this case), more precisely, 
              \begin{equation}\label{homogeneity}
                  \forall \;\lambda\in\mathbb{R}, \forall\; z\in Y, \quad J_q(\lambda z)=\operatorname{sgn } (\lambda)\cdot\vert\lambda\vert^{q-1}\cdot J_q(z).
              \end{equation}
              Equality \eqref{homogeneity} implies 
              \begin{equation*}
                  J_q\left(\operatorname{sgn } (\lambda) \vert\lambda\vert^\frac{1}{q-1} y\right)=\lambda J_q(y),
              \end{equation*}
              hence the function $s$ defined in \eqref{scalingprop} is given by $s(\lambda)=\operatorname{sgn }(\lambda)\vert\lambda\vert^{\frac{1}{q-1}}$ and satisfies $\lim\limits_{\alpha\rightarrow0} s(\alpha)=0$.
              Then,  source condition  \eqref{SC2} reads as
                \begin{equation*}
                    \exists\; p\de\in Y^*, v\de\in X \text{ such that } A^*p\de\in\partial R(u\de) \text{ and }p\de=-J_q(-Av\de),
                \end{equation*}
              or, equivalently,
                \begin{equation}\label{SClin}
                    \exists\; p\de\in Y^*, v\de\in X \text{ such that } A^*p\de\in\partial R(u\de) \text{ and } J_{q^*}(p\de)=Av\de,
                \end{equation}
                due to the reflexivity of $Y$, which implies $J_{q^*}=(J_q)^{-1}$.
            Thus, we recover the source condition from Lemma 5.1 in \cite{Grasmair13}.
            
            If $Y=L^2(\Gamma)$ and $q=2$, then \eqref{SClin} becomes $$\exists\; p\de\in L^2(\Gamma), v\de\in X \text{ such that } A^*p\de\in\partial R(u\de) \text{ and } p\de=Av\de,$$ i.e. $$\operatorname{Ran } (A^*A)\cap\partial R(u\de)\neq\emptyset,$$ which is precisely the source condition from \cite{Resmerita05}. 
        \end{remark}
        \begin{remark}
            The conclusion of Theorem \ref{thscaling} also holds in the case when $F$ is not necessarily Fr\'echet differentiable, as we can replace the derivative with the subdifferential. 
        \end{remark}
    \end{proof}

 \section{Kullback-Leibler Fidelity} \label{KLfidelity}
    This section aims to obtain higher order convergence rates  when the data fidelity  is the Kullback-Leibler functional and a source condition for the dual problem holds. As discussed in the first section of \cite{Rockafellar71}, when dealing with optimization problems it might be more convenient to work in $L^{\infty}$ spaces due to the straightforwardness of proving  continuity with respect to the uniform norm, as compared to the other $p-$norms. However, a drawback of using this space is the difficulty of applying duality arguments, since the information available on $(L^\infty)^*$ spaces is limited. Nevertheless, the analysis can be facilitated by employing  the formula for the conjugate of an integral functional defined on $L^{\infty}$, provided by Theorem 1 in \cite{Rockafellar71}. 
    
    Prior to presenting the results, it is necessary to briefly recall the structure of $(L^{\infty})^*$, as outlined in the beginning of the second section of \cite{Rockafellar71}. Every functional $z^*\in (L^{\infty}(\Gamma))^*$ can be uniquely written as 
    \begin{equation}\label{decomp}
        z^*=z^*_a+z^*_s, 
    \end{equation} 
    where $z^*_a$ and $z^*_s$ are the "absolutely continuous" and "singular" components, respectively. Moreover, there exists $\overline{z}\in L^1(\Gamma)$ such that $z^*_a$ corresponds to $\overline{z}$, as in 
    \begin{equation*}
        \forall\; z\in L^{\infty}(\Gamma),\quad z^*_a(z)=\int\limits_{\Gamma}z(t)\overline{z}(t)dt.
    \end{equation*}

    \subsection{Kullback-Leibler fidelity - exact data} \label{KLfidelityexact}

    We assume the following: 
    \begin{enumerate}[start=1,label={(P\arabic*)}]
    \item The operator $A:L^1(\Omega)\rightarrow L^{\infty}(\Gamma)$ is linear, compact and ill-posed.\label{P1} 
    %\item If $x\in L^1(\Omega)$ such that there exists $c_1>0$ with $c_1\leq x$ almost everywhere, then there exists $c_3>0$ such that $c_3\leq Ax$ almost everywhere on $\Gamma$. \label{P21}
    %\item If $x\in L^1(\Omega)$ such that there exists $c_2>0$ with $x\leq c_2$ almost everywhere, then there exists $c_4>0$ such that $Ax\leq c_4$ almost everywhere on $\Gamma$. \label{P22}
    \item The function $f\in L^{\infty}_+(\Gamma)$ is bounded away from $0$ almost everywhere on $\Gamma$, i.e., \label{P3}
    \begin{equation*}
        \exists\; c_1,c_2>0 \text{ such that } c_1\leq f(t)\leq c_2 \text{ a.e. } 
    \end{equation*}
    \end{enumerate}

    We will study the case when $H_f(z)=KL(f,z)$ for $z\in L^{\infty}(\Gamma)$,  given by \eqref{KLdef}. We introduce a novel source condition, actually on the dual problem \eqref{dualpb}, as mentioned at the beginning of this section:
    \begin{equation}\label{SC3} 
    \exists\; p\de\in L^{\infty}(\Gamma), v\de\in L^1(\Omega) \text{ such that } A^*p\de\in\partial R(u\de)\text{ and } fp\de=Av\de. 
    \end{equation}

    We can make the following observation, similarly to what was discussed in \cite{Resmerita05}. While $A^*$ maps $(L^{\infty}(\Gamma))^*$ into $L^{\infty}(\Omega)$, it is important to highlight that the previous condition is more restrictive, in the sense that it requires $p\de$ to be in the smaller space $L^{\infty}(\Gamma)\subset(L^{\infty}(\Gamma))^*$. 

    In order to be able to give an explicit formulation for \eqref{dualpb}, we need to focus on finding $H_f^*$, which is quite a technical task.
    
     For a nonempty and convex set $C$, denote by $\delta_C$ the indicator function of $C$, that is $$\delta_C(x)=\begin{cases}
                                                                                         0, &\text{ if }x\in C,\\
                                                                                         +\infty, &\text{ otherwise.}
                                                                                  \end{cases} $$ 
    Its Fenchel conjugate is the so-called support function of $C$ (see e.g. Example 4.3, Chapter I.4  in \cite{Ekeland99}), $$\delta_C^*(x^*)=\sup\{x^*(x)\vert x\in C\}.$$ 
     Let us recall some results from \cite{Rockafellar71}, which will be employed in the subsequent study.
     Assume that $\varphi:\Gamma\times\mathbb{R}\rightarrow(-\infty,+\infty]$ and consider functionals of the form 
     \begin{equation}\label{introckafellar}
     I_{\varphi}(z)=\int\limits_{\Gamma}\varphi(t,z(t))dt, z\in L,    
     \end{equation}
      where $L$ is a linear space of measurable functions from $\Gamma$ to $\mathbb{R}$.
     Assume, as in the first section of \cite{Rockafellar71}, that $\varphi$ is a normal convex integrand, i.e.,
     \begin{enumerate}
         \item for each $t\in\Gamma$, the mapping $\tau\mapsto\varphi(t,\tau)$ is lower semicontinuous, convex, and not identically $+\infty$, 
         \item for each $t\in\Gamma$, the interior of $D(t)=\{\tau\in\mathbb{R} \;\vert\; \varphi(t,\tau)<+\infty\}$ is nonempty, and for each $\tau\in\mathbb{R}$, $\varphi(t,\tau)$ is measurable in $t$.
     \end{enumerate}
     \begin{theo}[{\cite[443]{Rockafellar71}}][Theorem 1]\label{theorockafellar}
     Consider the functional given by \eqref{introckafellar} such that $\varphi$ is a normal convex integrand, $\varphi(t,z(t))$ is majorized by a summable function of $t$ for at least some $z\in L^{\infty}(\Gamma)$, and that the Fenchel conjugate $\varphi^*(t,z^*(t))$ is majorized by a summable function of $t$ for at least one $z^*\in L^1(\Gamma)$. Then $I_{\varphi}$ is well-defined on $L^{\infty}(\Gamma)$, and the convex function $I_{\varphi}^*$ on $(L^{\infty}(\Gamma))^*$ is given by 
     \begin{equation}\label{conjint}
         \forall \; z^*\in (L^{\infty}(\Gamma))^*, \quad I_{\varphi}^*(z^*)=I_{\varphi^*}(\overline{z})+\delta_C^*(z_s^*),
     \end{equation}
     where $\overline{z}\in L^1(\Gamma)$ corresponds to the "absolutely continuous" component of $z^*$, $z_s^*$ is the "singular" component of $z^*$, and $$C=\{z\in L^{\infty}(\Gamma) \;\vert\; I_{\varphi}(z) < +\infty\}.$$
     \end{theo}
     \begin{corol}[{\cite[445]{Rockafellar71}}][Corollary 1A]\label{cor1rockafellar}
     Assume, in addition to the hypothesis of Theorem 1, that the convex set $C$ is a cone. Then, in formula \eqref{conjint}, one has  
     \begin{equation*}
         I_{\varphi}^*(z^*)=\begin{cases}
                                I_{\varphi^*}(\overline{z}), & \text{ if }\;\forall\; z\in C,\quad z^*_s(\overline{z})\leq0,\\
                                +\infty, & \text{ otherwise.}
                            \end{cases}
     \end{equation*}
     \end{corol}
     \begin{corol}[{\cite[449]{Rockafellar71}}][Corollary 2C]\label{cor2rockafellar}
         Assume that there exists a function $\overline{u}\in L^{\infty}(\Gamma)$ and $r>0$ such that $$\forall\; x\in\mathbb{R} \text{ with } \vert x\vert<r,\; \varphi(t,\overline{u}(t)+x)\in L^1(\Gamma).$$ Then, for every such $\overline{u}$, a functional $z^*\in (L^{\infty}(\Gamma))^*$ is in $\partial I_{\varphi}(\overline{u})$ if and only if the "singular" component of $z^*$, namely $z_s^*$, vanishes, and $\overline{z} \in L^1(\Gamma)$ corresponding to $z_a^*$ satisfies $$\text{for almost every }t\in\Gamma, \overline{z}(t)\in\partial\varphi(t,\cdot)(\overline{u}(t)).$$
     \end{corol}

    The following lemma gives a formula for the conjugate of the Kullback-Leibler divergence regarded as a function of the second variable on $L^{\infty}(\Gamma)$.
    \begin{lemma}\label{conjlemma}
        Consider the function $H_f:L^{\infty}(\Gamma)\rightarrow[0,+\infty]$ given by 
        \begin{gather*}
            H_f(z)=
            \begin{cases}
                \int\limits_{\Gamma}\varphi_f(t,z(t)) dt, & \text{ if } \varphi_f(\cdot,z(\cdot))\in L^1(\Gamma),\\
                +\infty, & \text{ otherwise},
            \end{cases} 
        \end{gather*}
        with $\varphi_f:\Gamma\times\mathbb{R}\rightarrow(-\infty,+\infty],$
         \begin{equation}\label{phi}
            \varphi_f(t,\tau)= \begin{cases}
                        f(t)\text{ ln }\frac{f(t)}{\tau}-f(t)+\tau, & \text{ if }\tau\in(0,+\infty),\\
                        +\infty, & \text{ otherwise.}
                        \end{cases}
        \end{equation}
        Then the Fenchel conjugate of $H_f$ is  
        \begin{gather}\label{conjKL}
            H_f^*:(L^{\infty}(\Gamma))^*\rightarrow\mathbb{R}\cup\{+\infty\},  H_f^*(z^*)= \int\limits_{\Gamma}{\varphi_f^*(t,\overline{z}(t))} +\delta^*_{C}(z^*_s) = -\int\limits_{\Gamma}{f(t)\text{ ln}\left(1-\overline{z}(t)\right)dt}+\delta^*_{C}(z^*_s),
        \end{gather}
        where $\overline{z}\in L^1(\Gamma)$ corresponds to the "absolutely continuous" component of $z^*$, $z^*_s$ is the "singular" component of $z^*$, and $C$ is the set defined by
        \begin{equation}\label{setC}
            C=\{z\in L^{\infty}(\Gamma)\;\vert \;H_f(z)<+\infty\}.
        \end{equation}        
    \end{lemma}
    \begin{proof}
        Theorem \ref{theorockafellar} will be employed to determine the Fenchel conjugate of $H_f$. To this end, we need to prove that the hypotheses hold. 
        Firstly, the mapping $\tau\mapsto \varphi_f(t,\tau)$ is convex, and one can choose $f\in L^{\infty}(\Gamma)$ such that $\varphi_f(t,f(t))$ (which vanishes) is majorized by a summable function of $t$.
        
        Moving on to the Fenchel conjugate of $\tau\mapsto\varphi_f(t,\tau)$, it will be calculated using the definition:
        \begin{gather*}
            \varphi_f^*(t,\tau^*)=\sup_{\tau\in(0,+\infty)}\{\tau^*\tau-\varphi_f(t,\tau)\} =\sup_{\tau\in(0,+\infty)}\left\{\tau^*\tau-f(t)\text{ ln }f(t)+f(t)\text{ ln }\tau+f(t)-\tau\right\}.
        \end{gather*}
             If $\tau^*<1$, then the supremum is attained for $\tau=\frac{f(t)}{1-\tau^*}>0$. Otherwise, the mapping $\tau\mapsto \tau^*\tau-f(t)\text{ ln }f(t)+f(t)\text{ ln }\tau+f(t)-\tau$ is strictly increasing with 
             \begin{equation*}
                 \lim_{\tau\rightarrow+\infty}(\tau^*\tau-f(t)\text{ ln }f(t)+f(t)\text{ ln }\tau+f(t)-\tau)=+\infty.
             \end{equation*}
        Hence, for any $\tau^*<1,\; \varphi_f^*(t,\tau^*)=-f(t)\text{ ln }(1-\tau^*)$.
        
        Secondly, it is easy to see that the constant mapping $t\mapsto z_0(t)=0$ is in $L^1(\Gamma)$, and $\varphi_f^*(t,z_0(t))=0$ can be majorized by a function in $L^1(\Gamma)$ as well, implying that the hypotheses of the above-mentioned theorem hold. From this, one gets formula \eqref{conjKL} for the conjugate.
    \end{proof}
    
        Before deriving convergence rates in the Poisson noise framework, we express \eqref{dualpb} based on the considerations above. Consequently, the dual problem is the following,
    \begin{equation}\label{formdual}
        -\min\limits_{p\in (L^{\infty}(\Gamma))^*}\left\{-\frac{1}{\alpha}\int\limits_{\Gamma}{f(t)\text{ ln}\left[1+\alpha\overline{p}(t)\right]dt}+\frac{1}{\alpha}\delta^*_{C}(-\alpha p_s)+R^* (A^* p)\right\},
    \end{equation}
    where $\overline{p}\in L^1(\Gamma)$ corresponds to the "absolutely continuous" component of $p$ and $p_s$ is the "singular" component of $p$.

    We  focus on determining the subdifferential of $H_f$ at the points $z\in E$, where the set $E$ is given by \eqref{defE1}.
    \begin{prop}\label{propsubdiffconj}
        Let $z\in E$. Then the subdifferential of $H_f$ at $z$ is given by
        \begin{equation}\label{subdifffc}
            \partial H_f (z)=\left\{1-\frac{f}{z}\right\}.
        \end{equation}
    \end{prop}
    \begin{proof}
        In order to find the subdifferential of the function $H_f$, Corollary \ref{cor2rockafellar}  will be employed. From \eqref{decomp},  one gets that any $v\in (L^{\infty}(\Gamma))^*$ has the form $v=v_a+v_s$, where $v_a$ and $v_s$ are the "absolutely continuous" component and the "singular" component, respectively. 
        Verifying the  corollary's hypothesis amounts to showing that there exist $\overline{u}\in L^{\infty}(\Gamma)$ and $r>0$ such that $$\forall x\in\mathbb{R} \text{ with } \vert x \vert<r,\quad \varphi_f(t,\overline{u}(t)+x)\in L^1(\Gamma).$$ If this holds, then $v\in\partial H_f(\overline{u})$ if and only if $v_s=0$, and the element $v_a$ is $$\forall h\in L^{\infty}(\Gamma), \quad v_a(h)=\int\limits_{\Gamma}{(\varphi_f(t,\cdot))'(\overline{u}(t))h(t)dt}.$$ 
        
        Specifically, we claim that, for any fixed $z\in E$, the following holds, $$\forall x\in\mathbb{R}, \text{ with } \vert x \vert<\frac{1}{2}\cdot\operatorname{essinf}(z), \quad t\mapsto\varphi_f(t,z(t)+x)\in L^1(\Gamma).$$
        Indeed, for such $x$ and any $t\in\Gamma$,
        \begin{equation}\label{ineqinf}
            z(t)+x\geq\operatorname{essinf} (z)-\frac{1}{2}\cdot\operatorname{essinf}(z)=\frac{1}{2}\cdot\operatorname{essinf}(z)>0,
        \end{equation}
        where the last inequality holds due to the assumption $z\in E$. In addition to that, $z(\cdot)+x\in L^{\infty}(\Gamma)$ which,  in conjunction with \eqref{ineqinf}, leads to $z(\cdot)+x\in E$. 
        Consequently,
        $$\varphi_f(t,z(t)+x)=f(t)\text{ ln }\frac{f(t)}{z(t)+x}-f(t)+z(t)+x\in L^1(\Gamma).$$
        By employing the aforementioned result, one finds that for a $v\in (L^{\infty}(\Gamma))^*$, the following holds,
        \begin{equation*}
          v\in\partial H_f(z) \Leftrightarrow \forall h \in L^{\infty}(\Gamma),\quad v(h)=\int\limits_{\Gamma}{\left(1-\frac{f(t)}{z(t)}\right)h(t)dt}.
        \end{equation*} 
    \end{proof}

    \begin{remark}\label{remarkderiv}
        Under the assumptions of the previous result, one can also deduce that $H_f$ is G\^{a}teaux differentiable at $z\in E$ (see also \eqref{E1int}) and $H_f'(z)= 1-\frac{f}{z}$. Indeed, $z\in\operatorname{dom }H_f$ and $H_f$ is continuous at $z$ with respect to the topology induced by $\Vert\cdot\Vert_{\infty}$, so one can employ Proposition 5.3, p.23 from Chapter I.5 in \cite{Ekeland99}.
    \end{remark}

    The following proposition provides a subgradient of $H_f^*$ at each $z^*\in L^{\infty}(\Gamma)$ having the property that $1-z^*\in E$.
    \begin{prop}\label{prop2}
        Let $z^*\in L^{\infty}(\Gamma)$ such that $1-z^*\in E$. Then
        \begin{equation}\label{subdiffconj}
            \Phi\left(\frac{f}{1-z^*}\right)\in\partial H_f^*(z^*)
        \end{equation}
        holds, where $\Phi:L^{\infty}(\Gamma)\rightarrow(L^{\infty}(\Gamma))^{**}$ denotes the canonical embedding of $L^{\infty}(\Gamma)$ in its bidual.
    \end{prop}
    \begin{proof}
       Since $z^*\in L^{\infty}(\Gamma)$, then $z^*\in L^1(\Gamma)$ and, due to the uniqueness of  decomposition \eqref{decomp}, it can be regarded as an element of $(L^{\infty}(\Gamma))^*$ for which the singular component vanishes. Denote by $u:=\frac{f}{1-z^*}$, which is in $L^{\infty}(\Gamma)$ due to $1-z^*\in E$, and let $\zeta:=\Phi(u)\in (L^{\infty}(\Gamma))^{**}$. To establish the conclusion, it is enough to demonstrate that the following inequality holds,
       \begin{equation}\label{ineqsubdiff}
           \forall \;x^*\in\operatorname{ dom }H_f^*,\quad \langle\zeta,x^*-z^*\rangle\leq H_f^*(x^*)-H_f^*(z^*).
       \end{equation}
       
       Let us fix an $x^*\in\operatorname{ dom }H_f^*\subset (L^{\infty}(\Gamma))^*$ and consider its unique decomposition given by \eqref{decomp}, 
       \begin{equation*}
           x^*=x^*_a+x^*_s.
       \end{equation*}
       As mentioned earlier, there exists $\overline{x}\in L^1(\Gamma)$ such that $x^*_a$ corresponds to $\overline{x}$, as in 
    \begin{equation}\label{eq01}
        \forall y\in L^{\infty}(\Gamma), \quad x^*_a(y)=\int\limits_{\Gamma}y(t)\overline{x}(t)dt.
    \end{equation}
    With this decomposition and \eqref{conjKL}, one gets $$H_f^*(x^*)= \int\limits_{\Gamma}{\varphi_f^*(t,\overline{x}(t))} +\delta^*_{C}(x^*_s) = -\int\limits_{\Gamma}{f(t)\text{ ln}\left(1-\overline{x}(t)\right)dt}+\delta^*_{C}(x^*_s).$$
    We begin by evaluating the left-hand side of \eqref{ineqsubdiff},
    \begin{align}
        \langle\zeta,x^*-z^*\rangle &= \langle\Phi(u),x^*_a\rangle + \langle\zeta,x^*_s\rangle-\langle\Phi(u),z^*\rangle = x^*_a(u) + \langle\zeta,x^*_s\rangle-z^*(u) \nonumber\\
        &= \int\limits_{\Gamma}\overline{x}(t)\frac{f(t)}{1-z^*(t)}dt + \langle\zeta,x^*_s\rangle - \int\limits_{\Gamma}z^*(t)\frac{f(t)}{1-z^*(t)}dt, \label{lhsevall}
    \end{align}
    where the last equality follows from \eqref{eq01}.
    From the convexity of $w\mapsto-\operatorname{ ln }w,$ one obtains 
    \begin{equation}\label{ineqln}
            \forall w_1,w_2>0, \quad\frac{w_2-w_1}{w_2}\leq \operatorname{ ln }w_2-\operatorname{ ln }w_1.
    \end{equation}
    Because the convex set $C$ given by \eqref{setC} is a cone, Corollary \ref{cor1rockafellar} together with the condition $x^*\in\operatorname{dom }H_f^*$ shows that $-\int\limits_{\Gamma}f(t)\operatorname{ln }(1-\overline{x}(t))dt<+\infty$, hence $1-\overline{x}(t)>0$ almost everywhere.
    With this, \eqref{ineqln} reads as
    \begin{equation*}
        \text{for almost any }t\in\Gamma,\quad \frac{(1-z^*(t))-(1-\overline{x}(t))}{1-z^*(t)} \leq \operatorname{ln }(1-z^*(t))-\operatorname{ln }(1-\overline{x}(t)),
    \end{equation*}
    which together with assumption \ref{P3} leads to
    \begin{equation*}
        \text{for almost any }t\in\Gamma, \quad f(t)\operatorname{ln }(1-\overline{x}(t))\leq  -(\overline{x}(t)-z^*(t))\frac{f(t)}{1-z^*(t)} + f(t)\operatorname{ln }(1-z^*(t)),
    \end{equation*}
%    Since $x^*\in\operatorname{ dom }H_f^*$, one gets 
%    \begin{equation}\label{ineqq}
%            \int\limits_{\Gamma}f(t)\operatorname{ln }(1-\overline{x}(t))<+\infty.
%    \end{equation}
%    Indeed, since \eqref{ineqqq} holds and $$\int\limits_{\Gamma}[\overline{x}(t)-z^*(t)]\frac{f(t)}{1-z^*(t)}dt \in\mathbb{R}, \; \; \int\limits_{\Gamma}f(t)\operatorname{ln }(1-z^*(t)) \in\mathbb{R},$$
%    one can also conclude that \eqref{ineqq} holds.
    and
    \begin{equation}\label{intineq}
        \int\limits_{\Gamma}(\overline{x}(t)-z^*(t))\frac{f(t)}{1-z^*(t)}dt\leq \int\limits_{\Gamma}f(t)\operatorname{ln }(1-z^*(t)) - \int\limits_{\Gamma}f(t)\operatorname{ln }(1-\overline{x}(t)).
    \end{equation}
    On the other hand, $H_f\left(\frac{f}{1-z^*}\right)<+\infty$ leads to $u=\frac{f}{1-z^*}\in C$, hence $\langle\zeta,x^*_s\rangle=\langle\Phi(u),x^*_s\rangle= x^*_s(u) \leq \delta_C^*(x^*_s)$, which along with \eqref{lhsevall} and \eqref{intineq} demonstrates
    \begin{equation*}
        \langle\zeta,x^*-z^*\rangle\leq - \int\limits_{\Gamma}f(t)\operatorname{ln }(1-\overline{x}(t)) + \delta^*_C(x^*_s) - \left(-\int\limits_{\Gamma}f(t)\operatorname{ln }(1-z^*(t))\right),
    \end{equation*}
    that is precisely \eqref{ineqsubdiff}.
    \end{proof} 

    The previous results will be employed for deriving error estimates with respect to the Bregman distance. For this, we additionally make use of a strong convexity  property for an appropriate function.
    \begin{theo}\label{mainth}
        Let source condition \eqref{SC3} hold. Assume that $Au_{\alpha}\in E$, for any $\alpha>0$, and there exists a constant $m>0$ such that  $m\leq\inf\{\operatorname{essinf }Au_{\alpha} : \alpha>0, \alpha\text{ sufficiently small}\}$. Then the following inequality is true,
        \begin{equation}\label{estimmainth}
            \forall\alpha>0 \text{ sufficiently small}, \quad D_R(u_\alpha,u\de)\leq D_{R}(u\de-\alpha v\de, u\de)+O(\alpha^3).
        \end{equation}
    \end{theo}
    \begin{proof}
        To get \eqref{estimmainth}, we will show that \eqref{egalit} holds for $s(\alpha)=\alpha$ and $r=3$. More specifically, we focus on estimating the term $\langle\zeta_{\alpha}^{\dagger},p_{\alpha}-p\de\rangle$ of \eqref{egalitateBregman}. From that, one gets the conclusion by employing Theorem \ref{bregmanestim}.
        We start by pointing out appropriate subgradients from $\partial H_f^* (-\alpha p\de)$ and $\partial H_f^* (-\alpha p_{\alpha})$, which will be needed in the analysis. 

        According to \eqref{SC3}, $p\de\in L^{\infty}(\Gamma)$, therefore $\operatorname{esssup} \vert p\de\vert<+\infty$ and $\vert p\de(t)\vert \leq\operatorname{esssup} \vert p\de\vert$ almost everywhere on $\Gamma$. For $\alpha$ small enough, one gets $\alpha\vert p\de(t)\vert\leq\frac{1}{2}$ a.e. and 
        \begin{equation}\label{q1}
            0<\frac{1}{2}\leq1+\alpha p\de(t)\leq\frac{3}{2}\;\; \text{ a.e.} 
        \end{equation}
        It follows that $1+\alpha p\de\in E$, and thus \eqref{subdiffconj} yields
        \begin{equation*}
            \forall \alpha>0 \text{ small enough, } \quad\zeta_{\alpha}^{\dagger}=\Phi\left(\frac{f}{1+\alpha p\de}\right)\in\partial H_f^* (-\alpha p\de).
        \end{equation*}
        
        We would like to prove now that $1+\alpha p_{\alpha}\in E$. The function $H_f$ is proper, convex and continuous (and, in particular, lower semicontinuous). The Fenchel conjugate of a proper function is always convex and lower semicontinuous (see e.g. Part I, Chapter I.4 from \cite{Ekeland99}), hence $H_f^*$ inherits these properties. Moreover, it is easy to see that $H_f^*$ is proper as well. Theorem \ref{lemabenninggen} implies the following equivalent formulation for the optimality condition \eqref{OC2}, 
        \begin{equation*}
            \forall\alpha>0,\quad Au_\alpha\in\partial H_f^*(-\alpha p_{\alpha}) \Leftrightarrow-\alpha p_{\alpha}\in \partial H_f(Au_{\alpha}).
        \end{equation*}
        Since $Au_{\alpha}\in E$, due to the hypothesis, Proposition \ref{propsubdiffconj} yields
        \begin{equation}\label{eqz}
            \forall\alpha>0,\quad 1+\alpha p_{\alpha}=\frac{f}{Au_{\alpha}},
        \end{equation}
        which is  in $E$ (see also assumption \ref{P3}).
        Therefore we can employ again \eqref{subdiffconj} to find
         \begin{equation*}
            \forall \alpha>0, \quad\zeta_{\alpha}=\Phi\left(\frac{f}{1+\alpha p_{\alpha}}\right)\in \partial H_f^* (-\alpha p_{\alpha}).
        \end{equation*}
        
        The next step is to evaluate the following term from \eqref{egalit}, $$\langle\zeta_{\alpha}^{\dagger}, p_{\alpha}-p\de\rangle=\left\langle\Phi\left(\frac{f}{1+\alpha p\de}\right), p_{\alpha} -p\de\right\rangle=\int\limits_{\Gamma}{\frac{f(t)}{1+\alpha p\de(t)}(p_{\alpha}(t) -p\de(t))dt},$$ where the last equality holds because $p_{\alpha}, p\de$ are in $L^{\infty}(\Gamma)$.
        Using the Taylor expansion of $s\mapsto\frac{1}{1+s}$ around $0$, we have that for $s>0$, there exists $\theta\in[0,s]$ depending on $s$, such that
        \begin{equation*}
            \frac{1}{1+s}=1-s+\frac{1}{(1+\theta)^3}s^2,
        \end{equation*}
        which for $s=\alpha p\de(t)$ (with $t\in\Gamma$ arbitrary, but fixed) gives a $\theta(t)\in[0,\alpha p\de(t)]$ satisfying
        \begin{gather*}
            \frac{f(t)}{1+\alpha p\de(t)}=f(t)-\alpha f(t)p\de(t)+\alpha^2\frac{ f(t)(p\de(t))^2}{(1+\theta(t))^3}.
        \end{gather*}
        Using source condition \eqref{SC3}, one gets
        \begin{align}
            \int\limits_{\Gamma}{\frac{f(t)}{1+\alpha p\de(t)}(p_{\alpha}(t) -p\de(t))dt} &= \int\limits_{\Gamma}{f(t)(p_{\alpha}(t) -p\de(t))dt}-\alpha \int\limits_{\Gamma}{f(t)p\de(t)(p_{\alpha}(t) -p\de(t))dt} \nonumber \\
            &+\alpha^2 \int\limits_{\Gamma}{\frac{f(t)(p\de(t))^2}{(1+\theta(t))^3}(p_{\alpha}(t) -p\de(t))dt} \label{taylordezv}\\
           &= \int\limits_{\Gamma}{f(t)(p_{\alpha}(t) -p\de(t))dt}-\alpha \int\limits_{\Gamma}{Av\de(t)(p_{\alpha}(t) -p\de(t))dt}\nonumber\\
           &+\alpha^2 \int\limits_{\Gamma}{\frac{f(t)(p\de(t))^2}{(1+\theta(t))^3}(p_{\alpha}(t) -p\de(t))dt}. \nonumber
        \end{align}
        Next, we will prove that
        \begin{equation}\label{rate}
            \alpha^2 \int\limits_{\Gamma}{\frac{f(t)(p\de(t))^2}{(1+\theta(t))^3}(p_{\alpha}(t) -p\de(t))dt} - \frac{1}{\alpha} D_{H^*_{f}}^s(-\alpha p_{\alpha},-\alpha p\de)=O(\alpha^3).
        \end{equation}
        Firstly, we use the definition of the symmetric Bregman distance, assumption \ref{P3}, equality \eqref{eqz} and similar reasoning as in Proposition \ref{prop2} to get that for any $\alpha>0$ sufficiently small, 
%        \begin{equation}\label{inegsym}
%            \begin{gathered}
%                \frac{1}{\alpha} D_{H^*_{f}}^s(-\alpha p_{\alpha},-\alpha p\de)= \frac{1}{\alpha} \int\limits_{\Gamma}{\left(\frac{f(t)}{1+\alpha p_{\alpha}(t)}-\frac{f(t)}{1+\alpha p\de(t)}\right)(-\alpha p_{\alpha}(t)+\alpha p\de(t))dt}\\
%                 =\frac{1}{\alpha} \int\limits_{\Gamma}{\frac{f(t)}{(1+\alpha p_{\alpha}(t))(1+\alpha p\de(t))}(\alpha p_{\alpha}(t)-\alpha p\de(t))^2 dt} = \frac{1}{\alpha} \int\limits_{\Gamma}{\frac{Au_{\alpha}(t)}{1+\alpha p\de(t)}(\alpha p_{\alpha}(t)-\alpha p\de(t))^2 dt} \\ \geq \frac{1}{\alpha}\cdot\frac{2m}{3}\Vert\alpha p_{\alpha}-\alpha p\de\Vert_2^2,
%            \end{gathered}
%        \end{equation}             
            \begin{align}
                \frac{1}{\alpha} D_{H^*_{f}}^s(-\alpha p_{\alpha},-\alpha p\de)&=\frac{1}{\alpha} \int\limits_{\Gamma}{\frac{f(t)}{(1+\alpha p_{\alpha}(t))(1+\alpha p\de(t))}(\alpha p_{\alpha}(t)-\alpha p\de(t))^2 dt} \nonumber \\
                 &= \frac{1}{\alpha} \int\limits_{\Gamma}{\frac{Au_{\alpha}(t)}{1+\alpha p\de(t)}(\alpha p_{\alpha}(t)-\alpha p\de(t))^2 dt} \nonumber \\ 
                 &\geq \frac{1}{\alpha}\cdot\frac{2m}{3}\Vert\alpha p_{\alpha}-\alpha p\de\Vert_2^2=\frac{\gamma}{\alpha}\Vert\alpha p_{\alpha}-\alpha p\de\Vert_2^2, \label{inegsym}
            \end{align}
         with  $\gamma:= \frac{2m}{3}$, where the last inequality follows from \eqref{q1} and the hypothesis. 
        
         From the Cauchy-Schwarz and Young inequalities, it follows
        \begin{align*}
              \int\limits_{\Gamma}{\frac{f(t)(p\de(t))^2}{(1+\theta(t))^3}(p_{\alpha}(t) -p\de(t))dt} &\leq \left\Vert\frac{f(p\de)^2}{(1+\theta)^3}\right\Vert_2\cdot\Vert p_{\alpha}-p\de\Vert_2\leq \left\Vert f(p\de)^2\right\Vert_2\cdot\Vert p_{\alpha}-p\de\Vert_2 \\
              &\leq \alpha\frac{\Vert f(p\de)^2\Vert^2_2}{4\gamma} + \frac{\gamma\Vert p_{\alpha}-p\de\Vert^2_2}{\alpha},
        \end{align*} 
        which, in conjunction with \eqref{inegsym}, leads to
        \begin{align*}
            \alpha^2 \int\limits_{\Gamma}{\frac{f(t)(p\de(t))^2}{(1+\theta(t))^3}(p_{\alpha}(t) -p\de(t))dt} - \frac{1}{\alpha} D_{H^*_{f}}^s(-\alpha p_{\alpha},-\alpha p\de)= \alpha^3\frac{\Vert f(p\de)^2\Vert^2_2}{4\gamma} =O(\alpha^3).
        \end{align*}
        In conclusion, \eqref{rate} holds, which proves \eqref{egalit}.
    \end{proof}

\begin{remark}
The assumption in Theorem \ref{mainth}  on $Au_\alpha $ being away from zero is  essential in dealing with Kullback-Leibler functionals. This is satisfied  in regularization with various penalties, either naturally or enforced by additional constraints. See, e.g., Section \ref{KLfidelityjoint} for the case of joint Kullback-Leibler regularization, then the total variation minimization in \cite{bardsley2009total}, as well as the sparsity setting in \cite{figueiredo2010restoration}.
\end{remark}

\subsection{Kullback-Leibler fidelity - noisy  data} \label{KLfidelitynoisy}
    Assume now that \ref{P1} holds and that the exact data $f$ are not known, but we only have access to an approximation $f^{\delta}\in L^{\infty}_+(\Gamma)$ satisfying
    \begin{equation*}
        \Vert f-f^{\delta}\Vert_{\infty}\leq\delta,
    \end{equation*}
    and being bounded away from $0$ almost everywhere on $\Gamma$, i.e.,
    \begin{enumerate}[start=3,label={(P\arabic*)}]
        \item $\exists\; C_1,C_2>0 \text{ such that } C_1<f^{\delta}(t)<C_2 \text{ a.e. on } \Gamma$, uniformly with respect to $\delta>0$.\label{Pn}
    \end{enumerate}
    
    Let $H_{f^{\delta}}:L^{\infty}(\Gamma)\rightarrow[0,+\infty]$, 
        \begin{gather*}
            H_{f^{\delta}}(z)=
            \begin{cases}
                \int\limits_{\Gamma}\varphi_{f^{\delta}}(t,z(t)) dt, & \text{ if } z\in \operatorname{dom }g,\\
                +\infty, & \text{ otherwise},
            \end{cases} 
        \end{gather*}
        with $\varphi_{f^{\delta}}:\Gamma\times\mathbb{R}\rightarrow(-\infty,+\infty],$
         \begin{equation*}
            \varphi_{f^{\delta}}(t,\tau)= \begin{cases}
                        f^{\delta}(t)\text{ ln }\frac{f^{\delta}(t)}{\tau}-f^{\delta}(t)+\tau, & \text{ if }\tau\in(0,+\infty),\\
                        +\infty, & \text{ otherwise.}
                        \end{cases}
        \end{equation*}
    Let $u_{\alpha}^{\delta}$ and $p_{\alpha}^{\delta}$ be solutions for the primal and the dual problem, respectively. They must satisfy the conditions of the Fenchel-Rockafellar Theorem,
    \begin{gather}
        \xi_{\alpha}^{\delta}:=A^*p_{\alpha}^{\delta}\in\partial R(u_{\alpha}^{\delta}) \label{OC4},\\
        \zeta_{\alpha}^{\delta}:=Au_\alpha^{\delta}\in\partial H_{f^{\delta}}^*(-\alpha p_{\alpha}^{\delta}). \nonumber
    \end{gather}
     Assume again that $\partial H_{f^{\delta}}^*(-\alpha p\de)\neq\emptyset$ and  consider
    \begin{equation}\label{zeta}
        \zeta_{\alpha}^{\dagger}\in\partial H_{f^{\delta}}^*(-\alpha p\de).
    \end{equation}

    Auxiliary results similar to the ones in the exact data case hold in this setting, as well.
    \begin{lemma}\label{noisylemma}
    If \eqref{SC1} holds, then the following is true,
    \begin{equation}\label{noisylemmaeq}
        D_R^s(u_\alpha^{\delta},u\de)+\frac{1}{\alpha} D_{H^*_{f^{\delta}}}^s(-\alpha p_\alpha^{\delta},-\alpha p\de) = \langle f,p\de-p_{\alpha}^{\delta}\rangle + \langle \zeta_{\alpha}^{\dagger},p_{\alpha}^{\delta}-p\de\rangle,
    \end{equation}
    where $D_{H^*_{f^{\delta}}}^s(-\alpha p_\alpha^{\delta},-\alpha p\de)$ denotes the symmetric Bregman distance between $-\alpha p_\alpha^{\delta},-\alpha p\de$ with respect to the subgradients $\zeta_{\alpha}^{\delta}, \zeta^{\dagger}_{\alpha}$, respectively.
    \end{lemma}

    \begin{lemma}
        The Fenchel conjugate of $H_{f^{\delta}}$ is given by 
        \begin{align*}
            H_{f^{\delta}}^*:(L^{\infty}(\Gamma))^*\rightarrow\mathbb{R}\cup\{+\infty\},  H_{f^{\delta}}^*(z^*)= \int\limits_{\Gamma}{\varphi_{f^{\delta}}^*(t,\overline{z}(t))} +\delta^*_{C}(z^*_s)= -\int\limits_{\Gamma}{f^{\delta}(t)\text{ ln}\left(1-\overline{z}(t)\right)dt}+\delta^*_{C}(z^*_s),
        \end{align*}
        where $\overline{z}\in L^1(\Gamma)$ corresponds to the "absolutely continuous" component of $z^*$, $z^*_s$ is the "singular" component of $z^*$ and $$C=\{z\in L^{\infty}(\Gamma)\;\vert \;H_{f^{\delta}}(z)<+\infty\}.$$
    \end{lemma}
    \begin{prop}
        Let $z\in E$, where $E$ is defined by \eqref{defE1}. Then the subdifferential of $H_{f^{\delta}}$ at $z$ is
        \begin{equation}\label{subdifffcd}
            \partial H_{f^{\delta}} (z)=\left\{1-\frac{f^{\delta}}{z}\right\}.
        \end{equation}
    \end{prop}
    \begin{prop}
        Let $z^*\in L^{\infty}(\Gamma)$ such that $1-z^*\in E$. Then one has
        \begin{equation}\label{subdiffconjd}
            \Phi\left(\frac{f^{\delta}}{1-z^*}\right)\in\partial H_{f^{\delta}}^*(z^*),
        \end{equation}
        where $\Phi:L^{\infty}(\Gamma)\rightarrow(L^{\infty}(\Gamma))^{**}$ denotes the canonical embedding of $L^{\infty}(\Gamma)$ in its bidual.
    \end{prop}
    %The following result, similar to Theorem \ref{bregmanestim} also holds.
    %\begin{theo}
    %Let $r\in\mathbb{N}, r\geq1$. If \eqref{SC1} holds and there exists $v\in X$ and $s:[0,+\infty)\rightarrow[0,+\infty)$ such that 
    %\begin{equation} \label{egalit}
     %   \langle(H^*_{f^{\delta}})'(-\alpha p\de), p_{\alpha}^{\delta} -p\de\rangle - \frac{1}{\alpha} D_{H^*_{f^{\delta}}}^s(-\alpha p_\alpha^{\delta},-\alpha p\de) \leq \langle f,p_{\alpha}^{\delta}-p\de\rangle-s(\alpha)\langle Av, p_{\alpha}^{\delta}-p\de\rangle+ O(\alpha^r),
    %\end{equation}
    % then one gets the following inequality
    %\begin{equation}
    %    D_{R}(u_{\alpha}, u\de)\leq D_{R}(u\de-s(\alpha)v, u\de) + O(\alpha^r).
    %\end{equation}
    %Moreover, assuming that \eqref{egalit} is true without the term $O(\alpha^r)$, the conclusion also holds without this term, i.e.
    %\begin{equation}
    %    D_{R}(u_{\alpha}, u\de)\leq D_{R}(u\de-s(\alpha)v, u\de).
    %\end{equation}
    %\end{theo}

    We state below the error estimates in the case of noisy data.
    \begin{theo}\label{mainthnoise}
        Let source condition \eqref{SC3} hold. Assume that  $Au_{\alpha}^{\delta}\in E$, for any $\alpha, \delta>0$, and there exists $m>0$ with $m\leq\inf\{\operatorname{essinf }Au_{\alpha}^{\delta} : \alpha,\delta>0, \alpha, \delta \text{ sufficiently small}\}$. Then the following inequality holds,
        \begin{equation*}
            D_R(u_\alpha^{\delta},u\de) \leq   D_R(u\de-\alpha v\de, u\de) +\frac{\delta^2}{\alpha}\cdot\frac{\mu(\Gamma)}{\gamma} +\alpha\delta^2\cdot\frac{\Vert p\de\Vert^2_2}{\gamma} + \alpha^3\cdot\frac{\Vert f(p\de)^2\Vert^2_2}{\gamma}+\alpha^3\delta^2\cdot\frac{\Vert (p\de)^2\Vert_2^2}{\gamma},
        \end{equation*}
        where $\mu(\Gamma)$ stands for the Lebesgue measure of $\Gamma$ and $\gamma=\frac{2m}{3}$.
    \end{theo}
    \begin{proof}
        Following the steps of the proof of Theorem \ref{mainth}, one gets that for any $t\in\Gamma$ fixed, there exists a $\theta(t)\in[0,\alpha p\de(t)]$ such that
        \begin{equation*}
            \frac{f^{\delta}(t)}{1+\alpha p\de(t)}=f^{\delta}(t)- \alpha f^{\delta}(t)p\de(t)+ \alpha^2 \frac{f^{\delta}(t)(p\de(t))^2}{(1+\theta(t))^3}.
        \end{equation*}
            We need to evaluate $\langle \zeta_{\alpha}^{\dagger},p_{\alpha}^{\delta}-p\de \rangle = \left\langle  \Phi\left(\frac{f^{\delta}}{1+\alpha p\de}\right), p_{\alpha}^{\delta}-p\de\right\rangle$, which amounts to estimating
        \begin{equation*}
            \int\limits_{\Gamma}\frac{f^{\delta}(t)}{1+\alpha p\de(t)}(p_{\alpha}^{\delta}(t)-p\de(t))dt.
        \end{equation*}
        By arguing similarly as in the case of exact data, one obtains 
        \begin{align*}
            \int\limits_{\Gamma}\frac{f^{\delta}(t)}{1+\alpha p\de(t)}(p_{\alpha}^{\delta}(t)-p\de(t))dt &= \int\limits_{\Gamma} f^{\delta}(t)(p_{\alpha}^{\delta}(t) -p\de(t))dt - \alpha\int\limits_{\Gamma} Av\de(t)(p_{\alpha}^{\delta}(t) -p\de(t))dt\\
            &+ \alpha\int\limits_{\Gamma} (f(t)-f^{\delta}(t))p\de(t)(p_{\alpha}^{\delta}(t) -p\de(t))dt\\ 
            &+\alpha^2 \int\limits_{\Gamma}\frac{f^{\delta}(t)(p\de(t))^2}{(1+\theta(t))^3}(p_{\alpha}^{\delta}(t)-p\de(t))dt.
        \end{align*}

    By plugging this in  \eqref{noisylemmaeq}, it follows that
    \begin{align*}
        D_R(u_\alpha^{\delta},u\de)+\frac{1}{\alpha} D_{H^*_{f^{\delta}}}^s(-\alpha p_\alpha^{\delta},-\alpha p\de) &=-R(u\de)+R(u_{\alpha}^{\delta}) + \langle A^*p_{\alpha}^{\delta}, u\de-u_{\alpha}^{\delta}\rangle  -\alpha\langle A^*p_{\alpha}^{\delta},v\de\rangle\\
        &+ \alpha\langle A^*p\de, v\de\rangle- \int\limits_{\Gamma} (f(t)-f^{\delta}(t))(p_{\alpha}^{\delta}(t) -p\de(t))dt \\ 
        &+ \alpha\int\limits_{\Gamma} (f(t)-f^{\delta}(t))p\de(t)(p_{\alpha}^{\delta}(t) -p\de(t))dt \\
        &+\alpha^2 \int\limits_{\Gamma}\frac{f^{\delta}(t)(p\de(t))^2}{(1+\theta(t))^3}(p_{\alpha}^{\delta}(t)-p\de(t))dt.
    \end{align*}
    
    As in the proof of  Theorem \ref{mainth}, from \ref{Pn} we get that for  $\alpha>0$ small enough,
        \begin{align*}
            \frac{1}{\alpha} D_{H^*_{f^{\delta}}}^s(-\alpha p_{\alpha}^{\delta},-\alpha p\de) \geq \frac{1}{\alpha}\cdot\frac{2m}{3}\Vert\alpha p_{\alpha}^{\delta}-\alpha p\de\Vert_2^2=\frac{\gamma}{\alpha}\Vert\alpha p_{\alpha}^{\delta}-\alpha p\de\Vert_2^2.
        \end{align*}   

     By employing  Cauchy-Schwarz and Young inequalities, one gets
    \begin{align*}
        D_R(u_\alpha^{\delta},u\de) &\leq -R(u\de)+R(u_{\alpha}^{\delta}) + \langle A^*p_{\alpha}^{\delta}, u\de-\alpha v\de-u_{\alpha}^{\delta}\rangle+\alpha\langle A^*p\de,v\de\rangle\\
        &+\frac{\delta^2}{\alpha}\cdot\frac{\mu(\Gamma)}{\gamma} +\alpha\delta^2\cdot\frac{\Vert p\de\Vert^2_2}{\gamma} + \alpha^3\cdot\frac{\Vert f(p\de)^2\Vert^2_2}{\gamma}+\alpha^3\delta^2\cdot\frac{\Vert (p\de)^2\Vert_2^2}{\gamma}.
    \end{align*}
    
    From \eqref{OC4}, one obtains $\langle A^*p_{\alpha}^{\delta},u\de- \alpha v\de-u_{\alpha}^{\delta}\rangle\leq R(u\de- \alpha v\de)-R(u_{\alpha}^{\delta})$ and
    \begin{align*}
        D_R(u_\alpha^{\delta},u\de) &\leq  R(u\de- \alpha v\de)-R(u\de)+\alpha\langle Av\de,p\de\rangle\\ 
        &+\frac{\delta^2}{\alpha}\cdot\frac{\mu(\Gamma)}{\gamma} +\alpha\delta^2\cdot\frac{\Vert p\de\Vert^2_2}{\gamma} + \alpha^3\cdot\frac{\Vert f(p\de)^2\Vert^2_2}{\gamma}+\alpha^3\delta^2\cdot\frac{\Vert (p\de)^2\Vert_2^2}{\gamma}.
    \end{align*}
    Consequently, one has
    \begin{align*}
        D_R(u_\alpha^{\delta},u\de) &\leq   D_R(u\de-\alpha v\de, u\de) +\frac{\delta^2}{\alpha}\cdot\frac{\mu(\Gamma)}{\gamma} +\alpha\delta^2\cdot\frac{\Vert p\de\Vert^2_2}{\gamma}\\ &+ \alpha^3\cdot\frac{\Vert f(p\de)^2\Vert^2_2}{\gamma}+\alpha^3\delta^2\cdot\frac{\Vert (p\de)^2\Vert_2^2}{\gamma}.
    \end{align*}
\end{proof}

\begin{remark}
    In the context of Theorem \ref{mainthnoise}, if $R$ is assumed to be twice Fr\'echet differentiable, then $D_R(u\de-\alpha v\de,u\de)=O(\alpha^2)$ (see Remark \ref{remarkord}). In this case, the right hand side has the order $O(\alpha^2)+O(\frac{\delta^2}{\alpha})$.
    The choice of $\alpha$ is made such that $\alpha^2\sim\frac{\delta^2}{\alpha}$, i.e., $\alpha\sim\delta^{\frac{2}{3}}$, which gives the rate of convergence $O(\delta^{\frac{4}{3}})$. 
\end{remark}

\subsection{Joint additive Kullback-Leibler residual minimization} \label{KLfidelityjoint}

    Theorem \ref{mainth} is used in the sequel to establish higher order error estimates for the joint Kullback-Leibler  regularization.  
Namely, assume that minimization is performed over the feasible set $L^1_+(\Omega)$, and the penalty term $R$ is the Kullback-Leibler divergence between the variable $u$ and an a priori guess of the solution $u^*$, $$R(u)=KL(u,u^*).$$
This framework was dealt with in the paper \cite{Resmerita07}, which obtained error estimates under the weaker regularity assumption \eqref{SC1}.
In the current analysis we rely on the assumptions of \cite{Resmerita07}, which are are usually satisfied in practical situations such as emission tomography (see the Introduction of the aforementioned paper).

    \begin{corol}
        Assume that $A$ is a Fredholm integral operator of the first kind,
        \begin{equation*}
            \forall t\in\Gamma, \quad Au(t)=\int\limits_{\Omega}{a(t,s)u(s)ds},
        \end{equation*}
        such that the kernel $a(t,s)$ is bounded and bounded away from zero, i.e.
        \begin{equation*}
            \exists\; c_3,c_4>0 \text{ such that } c_3\leq a(t,s)\leq c_4, \text{ a.e.}
        \end{equation*}
        Then the following estimate holds,
        \begin{equation*}
            \forall\alpha>0,  \alpha\text{ sufficiently small}, \quad D_R(u_\alpha,u\de)\leq D_{R}(u\de-\alpha v\de, u\de)+O(\alpha^3).
        \end{equation*}
    \end{corol}
    \begin{proof}
        By employing the Corollary from p.1537 in \cite{Resmerita07}, one gets that for a fixed $\alpha>0$, $Au_{\alpha}$ is minorized and majorized by two positive constants almost everywhere on $\Omega$. It follows that $Au_{\alpha}\in E$.
        Then the proposition from p. 1540 in \cite{Resmerita07} gives the convergence $$\lim\limits_{\alpha\rightarrow0}\Vert u_{\alpha}-u\de\Vert_1=0, $$ which entails
        \begin{equation} \label{lim}
            \lim\limits_{\alpha\rightarrow0}\Vert u_{\alpha}\Vert_1=\Vert u\de\Vert_1.
        \end{equation} 
        If $\Vert u\de\Vert_1=0$, then $u\de=0$ holds a.e., yielding  $f=Au\de=0$ a.e., which contradicts assumption \ref{P3}. Therefore, $\Vert u\de\Vert_1>0$, and \eqref{lim} leads to
        $$\exists\; \alpha_0>0 \text{ such that } \forall \alpha\in(0,\alpha_0), \quad \Vert u_{\alpha}\Vert_1\geq \frac{\Vert u\de\Vert_1}{2}>0.$$
        Consequently, due to the non-negativity of the elements $u_{\alpha}$, one gets
        \begin{align*}
            \forall\alpha\in(0,\alpha_0), \forall t\in\Gamma,\quad Au_{\alpha}(t)&=\int\limits_{\Omega}{a(t,s)u_{\alpha}(s)ds} \geq c_3 \int\limits_{\Omega}{u_{\alpha}(s)ds}=c_3\Vert u_{\alpha}\Vert_1\\&\geq c_3\frac{\Vert u\de\Vert_1}{2}:=m>0.
        \end{align*}
        Therefore, the assumptions of Theorem \ref{mainth} hold and the conclusion follows.
    \end{proof}
    A similar estimate remains valid in the case of noisy data - see Theorem \ref{mainthnoise}.
    \begin{remark}\label{KLest}
         If $R(u) := KL(u, u^*)$, then $D_R(u\de - \alpha v\de, u\de)=KL(u\de-\alpha v\de, u\de)$.
        By employing Theorem 4.1 in \cite{Resmerita05}, based on a Taylor approximation, one gets the higher order error estimate $D_R(u_{\alpha},u\de)=O(\alpha^2)$.
    \end{remark}

    \subsection{Interpretation of the source condition} \label{KLfidelityinterpretation}
    Inspired by \cite{Benning}, we provide now an interpretation of source condition \eqref{SC3}. 
%    Assume, in addition, that 
%    \begin{enumerate}[start=3,label={(P\arabic*)}]
%        \item $\operatorname{Im }A\subset L^{\infty}(\Gamma)$ and $A:L^1(\Omega)\rightarrow L^{\infty}(\Gamma)$ is continuous. \label{P04}
%    \end{enumerate}

    \begin{prop} \label{interpretation}
    Let $\alpha >0$. 
    Source condition \eqref{SC3} is equivalent to the existence of a function $\overline{v}\in L^1(\Omega)$ such that $A\overline{v}>0$ a.e. and
    \begin{equation}\label{propc} 
        u\de\in\operatornamewithlimits{arg min}_{u\in L^1(\Omega)} \left\{\frac{1}{\alpha}KL(A\overline{v},Au)+R(u)\right\}.
    \end{equation}
    \end{prop}
    \begin{proof}
        Let $\alpha>0$. 
        Using the optimality conditions and the fact that the optimization problem is convex, relation \eqref{propc} is equivalent to
        \begin{equation} \label{minequiv}
             0\in \partial \left(\frac{1}{\alpha}KL(A\overline{v},A\cdot)+R(\cdot)\right)(u\de).
        \end{equation}
        
        According to the assumptions, one has $Au\de=f\in E$. It is easy to see that $E=\operatorname{int }_{L^{\infty}(\Gamma)}L^{\infty}_+(\Gamma)\subset\operatorname{ dom }KL(A\overline{v},\cdot)$, indicating that 
        \begin{equation}\label{interior}
            Au\de\in\operatorname{int }_{L^{\infty}(\Gamma)}\operatorname{ dom }KL(A\overline{v},\cdot).
        \end{equation}
         By employing the continuity of $A$, one gets $$u\de\in \operatorname{ dom } R\cap\operatorname{int }_{L^{1}(\Omega)}\operatorname{ dom }KL(A\overline{v},A\cdot).$$ 
        Therefore, equality holds in the sum rule for the subdifferential (see e.g. Proposition 5.6, p.26 from Part I, Chapter I in \cite{Ekeland99}), and \eqref{minequiv} is equivalent to
        \begin{equation}\label{sumrule}
            0\in \partial \left(\frac{1}{\alpha}KL(A\overline{v},A\cdot)\right)(u\de)+\partial R(u\de)=\frac{1}{\alpha} \partial \left(KL(A\overline{v},A\cdot)\right)(u\de)+\partial R(u\de).
        \end{equation}
        Because of \eqref{interior}, we get equality in the chain rule for the subdifferential (see e.g. Proposition 5.7 from Part I, Chapter I in \cite{Ekeland99}). More specifically, $$\partial \left(KL(A\overline{v},A\cdot)\right)(u\de)=A^*\left(\partial KL(A\overline{v}, \cdot)(Au\de)\right),$$ so \eqref{sumrule} becomes
        \begin{equation*}
            0\in \frac{1}{\alpha} A^*\left(\partial KL(A\overline{v}, \cdot)(Au\de)\right) + \partial R(u\de).
        \end{equation*}   

        Firstly, suppose that \eqref{SC3} holds.
        It suffices to prove that there exists $\overline{v}\in L^1(\Omega)$ such that $A\overline{v}>0$ a.e. and
        \begin{equation}\label{propc1}
            -\frac{1}{\alpha}A^*\left(1-\frac{A\overline{v}}{f}\right)\in\partial R(u\de) .
        \end{equation}
        From our assumption, we know that $$p\de=\frac{Av\de}{f}.$$ Since $A^*p\de\in\partial R(u\de)$, we also get $$A^*\left(\frac{Av\de}{f}\right)\in\partial R(u\de).$$ 
        Consider $\overline{v}=u\de+\alpha v\de\in L^1(\Omega)$. For $\alpha>0$ sufficiently small, one has $A\overline{v}=f+\alpha Av\de>0$ a.e. because of \ref{P3}. With this choice, \eqref{propc1} is true. 

        Assume now that there is a $\overline{v}\in L^1(\Omega)$ such that $A\overline{v}>0$ a.e. and that \eqref{propc1} holds. Then one has
        \begin{equation}\label{abc}
            A^*\left(\frac{1}{\alpha}\cdot\frac{A(\overline{v}-u\de)}{f}\right)\in\partial R(u\de).
        \end{equation}
        Let now $$v\de=\frac{1}{\alpha}(\overline{v}-u\de)\in L^1(\Omega)$$ and $$p\de=\frac{Av\de}{f}\in L^{\infty}(\Omega),$$ due to the assumption \ref{P3}. 
        From \eqref{abc}, we get $A^*p\de\in\partial R(u\de)$, which ends the proof.
    \end{proof}

\subsection{Variational source condition}\label{KL_VSC}
    As noted in \cite{Grasmair13}, a variational formulation for the source condition exhibits several advantages. For example, it would allow the derivation of an entire range of convergence rates. Motivated by this, we can state the following result, which replaces the aforementioned source condition \eqref{SC3} by a variational one.
    \begin{theo}
        Assume that $Au_{\alpha}\in E$ for any $\alpha>0$,  and  there exists a positive constant $m$ such that $m\leq\inf\{\operatorname{essinf }Au_{\alpha} : \alpha>0, \alpha\text{ sufficiently small}\}$.
        In addition to that, assume that there exist $p\de\in L^{\infty}(\Gamma)$ satisfying $A^*p\de\in\partial R(u\de)$, and a strictly increasing, continuous, concave function $\Phi:[0,+\infty)\rightarrow[0,+\infty)$, such that $\Phi(0)=0$ and the following variational source condition holds,
        \begin{equation}\label{varsc}
            \forall\; p\in L^{\infty}(\Gamma), \quad\int\limits_{\Gamma}{f(t)p\de(t)(p\de(t) -p(t))dt} \leq \Phi(D_{R^*}(A^*p,A^*p\de)).
        \end{equation}
        Then the following holds,
        \begin{equation}\label{concl}
            D_R(u_{\alpha},u\de) \leq 2\Psi(\alpha)+\Phi^{-1}\left(\alpha^2\frac{\Vert f(p\de)^2\Vert^2_2}{4\gamma}\right),    
        \end{equation}
         where $\Psi$ is the conjugate of the convex mapping $t\mapsto\Phi^{-1}(t)$ and $\gamma$ is chosen as in the proof of Theorem \ref{mainth}.
    \end{theo}
    \begin{proof}
        In order to show \eqref{concl}, we  evaluate the term $\langle\zeta_{\alpha}^{\dagger},p_{\alpha}-p\de\rangle$ of \eqref{egalitateBregman}.
        By plugging equality \eqref{taylordezv} in \eqref{egalitateBregman}, one gets
        \begin{align}
            D_R^s(u_\alpha,u\de)+\frac{1}{\alpha} D_{H^*_f}^s(-\alpha p_\alpha,-\alpha p\de) &= -\alpha \int\limits_{\Gamma}{f(t)p\de(t)(p_{\alpha}(t) -p\de(t))dt} \nonumber\\
            &+\alpha^2 \int\limits_{\Gamma}{\frac{f(t)(p\de(t))^2}{(1+\theta(t))^3}(p_{\alpha}(t) -p\de(t))dt}. \label{eqvar}
        \end{align}
        Assumption \eqref{varsc} and Young's inequality imply
        \begin{gather*}
            \alpha\int\limits_{\Gamma}{f(t)p\de(t)(p\de(t)-p_{\alpha}(t))dt} \leq \alpha \Phi(D_{R^*}(A^*p_{\alpha},A^*p\de))\leq \Psi(\alpha)+\Psi^*(\Phi(D_R(u\de,u_{\alpha})))\\ = \Psi(\alpha)+D_R(u\de,u_{\alpha})
        \end{gather*}
        and, similarly to \eqref{rate},
        \begin{align*}
            \alpha^2 \int\limits_{\Gamma}{\frac{f(t)(p\de(t))^2}{(1+\theta(t))^3}(p_{\alpha}(t) -p\de(t))dt} \leq \alpha^3\frac{\Vert f(p\de)^2\Vert^2_2}{4\gamma}+\frac{1}{\alpha} D_{H^*_{f}}^s(-\alpha p_{\alpha},-\alpha p\de).
        \end{align*}
        For $\alpha>0$ small enough, one has 
        \begin{equation*}
            \alpha^3\frac{\Vert f(p\de)^2\Vert^2_2}{4\gamma} = \alpha\cdot\alpha^2\frac{\Vert f(p\de)^2\Vert^2_2}{4\gamma} \leq\Psi(\alpha)+\Psi^*\left(\alpha^2\frac{\Vert f(p\de)^2\Vert^2_2}{4\gamma}\right),
        \end{equation*}
        which, combined with \eqref{eqvar}, gives the conclusion.
    \end{proof}
    
    The following result takes into account noisy data, too.
    \begin{theo}
        Assume that $Au_{\alpha}^{\delta}\in E$  for any $\alpha, \delta>0$,  and  there exists a positive constant $m$ with $m\leq\inf\{\operatorname{essinf }Au_{\alpha}^{\delta} : \alpha,\delta>0, \alpha, \delta \text{ sufficiently small}\}$.
        In addition, assume that there exist $p\de\in L^{\infty}(\Gamma)$ with $A^*p\de\in\partial R(u\de)$,  and a strictly increasing, continuous, concave function $\Phi:[0,+\infty)\rightarrow[0,+\infty)$,  such that $\Phi(0)=0$ and 
        \begin{equation*}
            \forall p\in L^{\infty}(\Gamma),\quad \int\limits_{\Gamma}{f(t)p\de(t)(p\de(t) -p(t))dt} \leq \Phi(D_{R^*}(A^*p,A^*p\de)).
        \end{equation*}
        Then the following holds, $$D_R(u_{\alpha}^\delta,u\de) \leq \frac{\delta^2}{\alpha}\cdot\frac{\mu(\Gamma)}{\gamma} + 2\Psi(\alpha)+\Phi^{-1}\left(\delta^2\frac{\Vert p\de\Vert_2^2}{\gamma}+\alpha^2\delta^2\frac{\Vert(p\de)^2\Vert_2^2}{\gamma}+\alpha^2\frac{\Vert f(p\de)^2\Vert^2_2}{\gamma}\right),$$ where $\Psi$ is the conjugate of the convex mapping $t\mapsto\Phi^{-1}(t)$ and $\gamma$ was chosen as in Theorem \ref{mainthnoise}.
    \end{theo}
    \begin{proof}
        The proof is similar to the one for the exact data setting, and relies on employing Young's inequality several times.
    \end{proof}

    The next statement follows the idea of Lemma 5.1 in \cite{Grasmair13} and provides a connection between the variational source condition from the previous theorems and the operator range source condition. However, the latter refers to the range of $A^{**}$, instead of $A$, since we work in non-reflexive Banach spaces.
    \begin{theo}\label{connection}
        Assume  there exists $p\de\in L^{\infty}(\Gamma)$ such that $A^*p\de\in\partial R(u\de)$. Let  $\Phi:[0,+\infty)\rightarrow[0,+\infty)$ be a strictly increasing, continuous, concave  function such that $\Phi(0)=0$ and 
        \begin{equation*}
            \forall p\in L^{\infty}(\Gamma), \quad\int\limits_{\Gamma}{f(t)p\de(t)(p\de(t) -p(t))dt} \leq \Phi(D_{R^*}(A^*p,A^*p\de)).
        \end{equation*}
        If $R^*$ is twice Fr\'{e}chet differentiable at $A^*p\de$ and $\Phi(t)\sim t^{\frac{1}{2}}$ as $t\rightarrow0$, then $fp\de\in \operatorname{Ran }(A^{**})$ holds.
    \end{theo}
    \begin{proof}
        The Fr\'{e}chet differentiability of $R^*$ yields, for any $p\in (L^{\infty}(\Gamma))^*$,
        \begin{align*}
            D_{R^*}(A^*p,A^*p\de)&=R^*(A^*p)-R^*(A^*p\de)-\langle u\de,A^*(p-p\de)\rangle\\
            &=(R^*)''(A^*p\de)(A^*(p-p\de),A^*(p-p\de))+o(\Vert A^*(p-p\de)\Vert^2).
        \end{align*}
        Denote $\Tilde{p}:=p\de-p$. The previous equality and the variational inequality above imply
        \begin{gather*}
            \int\limits_{\Gamma}{f(t)p\de(t)\Tilde{p}(t)dt} \leq \Phi((R^*)''(A^*p\de)(A^*\Tilde{p},A^*\Tilde{p})+o(\Vert A^*\Tilde{p}\Vert^2)).
        \end{gather*}
        This, together with $(R^*)''(A^*p\de)(A^*\Tilde{p},A^*\Tilde{p})\leq\Vert (R^*)''(A^*p\de)\Vert \Vert A^*\Tilde{p}\Vert^2$ and the monotonicity of $\Phi$, leads to
        \begin{equation*}
            \int\limits_{\Gamma}{f(t)p\de(t)\Tilde{p}(t)dt} \leq \Phi(\Vert (R^*)''(A^*p\de)\Vert \Vert A^*\Tilde{p}\Vert^2+o(\Vert A^*\Tilde{p}\Vert^2)).
        \end{equation*}
        If $\Phi(t)\sim t^{\frac{1}{2}}$ as $t\rightarrow0$, then there exists $\Tilde{C}>0$ such that for any $\Vert A^*\Tilde{p}\Vert$ sufficiently small,
        \begin{equation*}
            \int\limits_{\Gamma}{f(t)p\de(t)\Tilde{p}(t)dt} \leq \Tilde{C}\Vert (R^*)''(A^*p\de)\Vert^{\frac{1}{2}} \Vert A^*\Tilde{p}\Vert.
        \end{equation*}
        
        In order to prove that this holds in fact for any $p\in L^{\infty}(\Gamma)$, consider $\epsilon>0$ small enough such that the previous inequality holds with $A^*\Tilde{p}$ replaced by $A^*(\epsilon p)$. By multiplying it with $\frac{1}{\epsilon}$ and using the positive homogeneity of both sides, the following holds,
        \begin{equation*}
            \forall p\in (L^{\infty}(\Gamma))^*,\quad \int\limits_{\Gamma}{f(t)p\de(t)p(t)dt} \leq \Tilde{C}\Vert (R^*)''(A^*p\de)\Vert^{\frac{1}{2}} \Vert A^*p\Vert.
        \end{equation*}
        Lemma 8.21 from \cite{Scherzer08} implies that this is equivalent to $fp\de\in\operatorname{Ran }(A^{**})$.
    \end{proof}

\section{Convex integrands as data fidelities} \label{generalfidelities}
    Our goal here is to derive higher order error estimates for variational regularization formulated with a more general convex non-quadratic data fidelity. In order to pursue this, we unveil a stronger source condition that turns out to relate a solution of the original problem to both ranges of the adjoint operator and of the operator itself, namely, involving both primal and dual problems. This formulation is consistent with the previous finding, reducing to \eqref{SC3} in the case when the data fidelity is the Kullback-Leibler functional. Additionally, we establish similar error estimates  under a variational source condition - compare to the particular setting of Subsection \ref{KL_VSC}. For the sake of simplicity, we perform the analysis in the exact data case, since the slightly different context of noisy data can be dealt with analogously.
    
    Consider  the framework of a more general convex data fidelity function $H_f$ defined by 
    \begin{equation}\label{defhf}
            H_f(z)=\int\limits_{\Gamma}\varphi_f(t,z(t)) dt,
    \end{equation}
    where $\varphi_f:\Gamma\times\mathbb{R}\rightarrow(-\infty,+\infty]$ is a proper, convex and lower semicontinuous function.  
    Suppose that the following source condition holds,
     \begin{gather}
            \exists\;p\de\in Y^* \text{ single-valued and bounded a.e.}, \exists\; v\de\in X \text{ with } A^*p\de\in\partial R(u\de) \text{ and } \nonumber\\ \text{ for almost any } t\in\Gamma, \;
            p\de(t)((\varphi_f(t,\cdot))^*)''(0)=Av\de(t). \label{SC5}
    \end{gather}
     We will reformulate \eqref{SC5} in a way that makes use of the second order derivative of $\varphi_f(t,\cdot)$. This approach will benefit from avoiding calculating the Fenchel conjugate near the origin, which can be hard to achieve in several situations (e.g., for the Itakura-Saito distance).
     
    Additionally, suppose that there exists $\zeta_{\alpha}^{\dagger}\in\partial H_f^*(-\alpha p\de)$ such that 
    \begin{equation}\label{gradient}
         \langle\zeta_{\alpha}^{\dagger},p_{\alpha}-p\de\rangle=\int\limits_{\Gamma}((\varphi_f(t,\cdot))^*)'(-\alpha p\de(t))(p_{\alpha}(t)-p\de(t))dt.
    \end{equation}

    \begin{theo}\label{generalcase}
        Assume that $\varphi_f$ is chosen such that, for almost any $t\in\Gamma$,
         \begin{enumerate}[start=1,label={(C\arabic*)}]
            \item $\varphi_f(t,\cdot)$ is three times differentiable on a neighborhood of $f(t)$, \label{a1}
            \item $(\varphi_f(t,\cdot))'(f(t))=0$, $(\varphi_f(t,\cdot))''(f(t))\neq0$, \label{a2}
            \item There exists $\Tilde{C}$ such that, for almost all $t\in\Gamma$, there exists a neighborhood $U(t)$ of $f(t)$ having the property that $\frac{(\varphi_f(t,\cdot))'''(z)}{((\varphi_f(t,\cdot))'')^3(z)}$ is uniformly bounded by $\Tilde{C}$ on $U(t)$, i.e., $$\forall z\in U(t), \quad \left\vert\frac{(\varphi_f(t,\cdot))'''(z)}{((\varphi_f(t,\cdot))'')^3(z)}\right\vert\leq\Tilde{C}$$. \label{a3}
        \end{enumerate}
          Moreover, assume  that
    \begin{equation}\label{ineqbd}
            \text{there exists } \;C>0 \text{ such that } \forall\;\alpha>0 \text{ sufficiently small, }  D_{H_f^*}^s(-\alpha p_{\alpha}, -\alpha p\de)\geq C\alpha^2\Vert p_{\alpha}-p\de\Vert_2^2
        \end{equation}
        and that the source condition \eqref{SC5} for the dual problem holds, which can also be expressed as
        \begin{gather}
            \exists\;p\de\in Y^* \text{ single-valued and bounded a.e.}, \exists\; v\de\in X \text{ with } A^*p\de\in\partial R(u\de) \text{ and }\nonumber\\ \text{ for almost any } t\in\Gamma,   \;
            \frac{p\de(t)}{(\varphi_f(t,\cdot))''(f(t))}=Av\de(t). \label{SC4}
        \end{gather}
        Then the following estimate holds, 
        $$\forall \; \alpha > 0 \text{ sufficiently small}, \quad D_R(u_{\alpha},u\de)\leq D_R(u\de-\alpha v\de, u\de)+O(\alpha^{3}).$$
    \end{theo}
    \begin{proof}
            We seek to evaluate $\langle\zeta_{\alpha}^{\dagger},p_{\alpha}-p\de\rangle$ with the aim of obtaining the sum of $\langle f,p_{\alpha}-p\de\rangle$, a first order term (with respect to $\alpha$), for which the source condition will be used, and a remainder of second order. 

            For any $t\in\Gamma$, let us make the following notation $\psi_f(t,\cdot):=(\varphi_f(t,\cdot))'$.
            According to \ref{a2}, $\psi_f(t,\cdot)(f(t))=0$, and $\psi_f(t,\cdot)'(f(t))\neq0$, one can employ the Implicit Function Theorem to get the existence of an open set $U\subset\mathbb{R}$ such that $0\in U$ and $(\psi_f(t,\cdot))^{-1}$ exists and is differentiable on $U$ with 
        \begin{equation}\label{IFT}
         \forall \;\tau\in U,\quad ((\psi_f(t,\cdot))^{-1})'(\tau)=\frac{1}{(\psi_f(t,\cdot))'(\nu)}, 
        \end{equation}
        where $\nu=(\psi_f(t,\cdot))^{-1}(\tau)$.
            The previous point, along with Lemma \ref{lemabenning}, implies that $(\varphi_f(t,\cdot))^*$ is differentiable on $U$, and 
            \begin{equation}\label{derivp}
                ((\varphi_f(t,\cdot))^*)'=(\psi_f(t,\cdot))^{-1}.
            \end{equation}

            Moreover, by employing \eqref{IFT}, one can also compute $(((\varphi_f(t,\cdot))')^{-1})''$ in $U$, i.e.,
        \begin{gather*} 
            \forall\; \tau\in U,\quad ((\psi_f(t,\cdot))^{-1})''(\tau)=-\frac{\psi_f(t,\cdot)''(\nu)\cdot((\psi_f(t,\cdot))^{-1})'(\tau)}{((\psi_f(t,\cdot))'(\nu))^2}  =-\frac{\psi_f(t,\cdot)''(\nu)}{((\psi_f(t,\cdot))'(\nu))^3},
        \end{gather*}
        which also shows that $(\varphi_f(t,\cdot))^*$ is three times locally differentiable, and 
            \begin{equation} \label{derivord3}
                \forall\; \tau\in U, \quad  ((\varphi_f(t,\cdot))^*)'''(\tau) = -\frac{\psi_f(t,\cdot)''(\nu)}{((\psi_f(t,\cdot))'(\nu))^3}.
            \end{equation}
        
        Since $p\de$ is assumed to be bounded, then for any $t$ fixed and for $\alpha$ sufficiently small, $-\alpha p\de(t)$ is in a neighborhood of $0$. The Taylor expansion for $s\mapsto ((\varphi_f(t,\cdot))^*)'(s)$ around $0$ gives the existence of $\theta(t)$ on the line segment between $0$ and $-\alpha p\de(t)$ (where $t\in\Gamma$ fixed) such that 
        \begin{align*}
            ((\varphi_f(t,\cdot))^*)'(-\alpha p\de(t))&=((\varphi_f(t,\cdot))^*)'(0)-\alpha p\de(t)[(\varphi_f(t,\cdot))^*]''(0) \\
            &+\frac{1}{2}\alpha^2(p\de(t))^2((\varphi_f(t,\cdot))^*)'''(\theta(t)).
        \end{align*}
        We now evaluate the terms of the right-hand side individually. 
        By employing \eqref{derivp}, the first one becomes $((\varphi_f(t,\cdot))^*)'(0)=(\psi_f(t,\cdot))^{-1}(0)$, which together with the condition $\psi(t,\cdot)(f(t))=0$ implies $$((\varphi_f(t,\cdot))^*)'(0)=f(t).$$ 
        For the second term, we make use of the source condition \eqref{SC4}. 
        One has 
        \begin{equation}\label{53}
            ((\varphi_f(t,\cdot))^*)''(0)=((\psi_f(t,\cdot))^{-1})'(0).
        \end{equation}         
        Equality \eqref{IFT} implies 
        \begin{equation*}
            ((\psi_f(t,\cdot))^{-1})'(0)=\frac{1}{(\psi_f(t,\cdot))'(f(t))},
        \end{equation*}
        and this, together with \eqref{SC4} and \eqref{53}, gives $$-\alpha p\de(t)((\varphi_f(t,\cdot))^*)''(0)=-\alpha Av\de(t).$$ 
        From the argument above, one can also notice the equivalence between the two formulations for the source condition, namely \eqref{SC5} and \eqref{SC4}.
        
        According to \eqref{derivord3}, the last term of the Taylor expansion becomes
        \begin{equation*}
            ((\varphi_f(t,\cdot))^*)'''(\theta(t)) = -\frac{(\varphi_f(t,\cdot))'''(\nu)}{((\varphi_f(t,\cdot))''(\nu))^3},
        \end{equation*}
        where $\nu:=((\varphi_f(t,\cdot))')^{-1}(\theta(t))$. When $\alpha$ approaches zero, $\theta(t)$ is arbitrarily close to the origin, and since $(\varphi_f(t,\cdot))')^{-1}$ is differentiable, one gets, in particular, that it is continuous, which shows that $\nu$ gets arbitrarily close to $f(t)$.
        Consequently, \ref{a3} implies that $\vert((\varphi_f(t,\cdot))^*)'''(\theta(t))\vert\leq\Tilde{C}$.

        From all the above,
        \begin{align*}
            \langle\zeta_{\alpha}^{\dagger},p_{\alpha}-p\de\rangle&=\int\limits_{\Gamma} f(t)(p_{\alpha}(t)-p\de(t))dt -\alpha \int\limits_{\Gamma} Av\de(t)(p_{\alpha}(t)-p\de(t))dt \\&+ \frac{1}{2}\alpha^2\int\limits_{\Gamma}(p\de(t))^2((\varphi_f(t,\cdot))^*)'''(\theta(t))(p_{\alpha}(t)-p\de(t))dt.
        \end{align*}

        As in the proof of Theorem \ref{mainth}, the last term can be majorized 
        as follows 
        \begin{align*}
            &\frac{1}{2}\alpha^2\int\limits_{\Gamma}(p\de(t))^2((\varphi_f(t,\cdot))^*)'''(\theta(t))(p_{\alpha}(t)-p\de(t))dt \leq \frac{\alpha^{3}}{8C}\Tilde{C}^2\Vert(p\de)^2\Vert^2 + \alpha C\Vert p_{\alpha}-p\de\Vert^2.
 %           = \frac{1}{\alpha} \int\limits_{\Gamma}\frac{1}{2}\alpha^2(p\de(t))^2[(\varphi_f(t,\cdot))^*]'''(\theta(t))\cdot\alpha(p_{\alpha}(t)-p\de(t))dt\\
 %           = \frac{1}{\alpha} \int\limits_{\Gamma}\frac{1}{2}\frac{\alpha^2}{\sqrt{2C}}(p\de(t))^2[(\varphi_f(t,\cdot))^*]'''(\theta(t))\cdot \sqrt{2C}\alpha(p_{\alpha}(t)-p\de(t))dt\\
 %           \leq \frac{1}{\alpha} \int\limits_{\Gamma}\left[\frac{1}{8}\frac{\alpha^{4}}{Cq}(p\de(t))^{4}\vert[(\varphi_f(t,\cdot))^*]'''(\theta(t))\vert^2+\frac{1}{2} C2\alpha^2\vert p_{\alpha}(t)-p\de(t)\vert^2\right]dt\\ 
 %          \leq\frac{1}{\alpha}\left[ \frac{1}{8}\frac{\alpha^{4}}{Cq}\Tilde{C}\Vert(p\de)^2\Vert^2 + \frac{1}{2}C2\alpha^2\Vert p_{\alpha}-p\de\Vert^2 \right]\\
 %           = \frac{1}{8}\frac{\alpha^{3}}{Cq}\Tilde{C}^2\Vert(p\de)^2\Vert^2 + C\alpha\Vert p_{\alpha}-p\de\Vert^2.
        \end{align*}

     From \eqref{ineqbd}, we have
        \begin{align*}
            \langle\zeta_{\alpha}^{\dagger},p_{\alpha}-p\de\rangle-\frac{1}{\alpha} D_{H^*_f}^s(-\alpha p_\alpha,-\alpha p\de) &\leq \int\limits_{\Gamma} f(t)(p_{\alpha}(t)-p\de(t))dt -\alpha \int\limits_{\Gamma} Av\de(t)(p_{\alpha}(t)-p\de(t))dt \\
           &+ \frac{\alpha^{3}}{8C}\Tilde{C}^2\Vert(p\de)^2\Vert^2 + C\alpha\Vert p_{\alpha}-p\de\Vert^2 -C\alpha\Vert p_{\alpha}-p\de\Vert^2.
        \end{align*}
        
        Therefore, by employing Theorem \ref{bregmanestim} we get the conclusion.
    \end{proof}
    \begin{remark}
        In the more general case when condition \eqref{ineqbd} from Theorem \ref{generalcase} is replaced by 
        \begin{equation*}
            \exists\; C>0, q>1 \text{ such that }  \forall\;\alpha>0 \text{ sufficiently small, }  D_{H_f^*}^s(-\alpha p_{\alpha},-\alpha p\de)\geq C\alpha^q\Vert p_{\alpha}-p\de\Vert^q,
        \end{equation*}
        one gets the following estimate,
        \begin{equation*}
        D_R(u_{\alpha},u\de)\leq D_R(u\de-\alpha v\de, u\de)+\alpha^{\frac{q+1}{q-1}}\cdot\frac{q-1}{q}\cdot\frac{1}{C^{\frac{1}{q-1}}q^{\frac{1}{q-1}}2^{\frac{q}{q-1}}}\cdot\Tilde{C}^{\frac{q}{q-1}}\Vert (p\de)^2\Vert^{\frac{q}{q-1}}.    
        \end{equation*}
        
    \end{remark}
    \begin{remark}
        We will show now that source condition \eqref{SC4} reduces to the previously known formulation in the case when the data fidelity is the Kullback-Leibler divergence. Moreover, we present formally the source condition  in the case of the Itakura-Saito distance, which is frequently employed in variational models for reconstructing images corrupted by noise following a gamma distribution (see papers \cite{aubert2008variational}, \cite{darbon2022hamilton} and \cite{barnett2023multiscale}). 
        \begin{enumerate}
            \item Consider the context of Section \ref{KLfidelity}. The proof of Theorem \ref{mainth} implies that \eqref{gradient} holds, $\varphi_f$ given by \eqref{phi} is three times differentiable with respect to the second variable on $(0,+\infty)$, and for almost any $t\in\Gamma$ fixed, $(\varphi_f(t,\cdot))'(f(t))=1-\frac{f(t)}{f(t)}=0$, and $(\varphi_f(t,\cdot))''(f(t))=\frac{f(t)}{(f(t))^2}=\frac{1}{f(t)}$. Additionally, $\frac{(\varphi_f(t,\cdot))'''}{((\varphi_f(t,\cdot))'')^3}(\tau)=-\frac{2\tau^3}{(f(t))^2}$ is bounded on a neighborhood of $f(t)$, whenever $f$ is a bounded function, and \eqref{ineqbd} was shown to hold in \eqref{inegsym}.
            Assumption \eqref{SC4} becomes 
            \begin{gather*}
                \exists\; p\de\in(L^{\infty}(\Gamma))^* \text{ single-valued and bounded a.e.}, \exists\; v\de\in L^1(\Omega) \text{ such that }\\ A^*p\de\in\partial R(u\de) \text{ and } fp\de=Av\de \text{ a.e.},
            \end{gather*}
            which coincides with \eqref{SC3}.
            \item  If $Y=L^{\infty}(\Gamma)$ and $H_f(z)$ is the Itakura-Saito distance, i.e., $$H_f(z)=\int\limits_{\Gamma}\left( \operatorname{ln}\frac{z(t)}{f(t)}+\frac{f(t)}{z(t)}-1 \right)dt, $$ then $\varphi_f(t,\tau)=\operatorname{ln}\frac{\tau}{f(t)}+\frac{f(t)}{\tau}-1$. Since the mapping $H_f$ is not globally convex, let us assume that $\tau\in(0,2\cdot\inf\{f(t) \vert t\in\Gamma\})$, as in \cite{aubert2008variational}, so that $\tau\mapsto\varphi_f(t,\tau)$ is convex. For almost any $t\in\Gamma$, $\varphi_f(t,\cdot)$ is three times differentiable, satisfying   $(\varphi_f(t,\cdot))'(f(t))=0, (\varphi_f(t,\cdot))''(f(t))\neq 0$ and
            \begin{align*}
                \frac{(\varphi_f(t,\cdot))'''}{((\varphi_f(t,\cdot))'')^3}(\tau)=\frac{2\tau^5(\tau-3f(t))}{(2f(t)-\tau)^3},
            \end{align*}
            where the last expression is bounded for $\tau$ in a neighborhood of $f(t)$.
            In order to check \eqref{gradient} and \eqref{ineqbd}, additional assumptions are needed.
             In this situation, \eqref{SC4} translates to 
        \begin{gather*}
            \exists\; p\de\in Y^* \text{ single-valued and bounded a.e.}, \exists\; v\de\in X \text{ such that } A^*p\de\in\partial R(u\de) \text{ and }\\  f^2p\de=Av\de \text{ a.e.}
        \end{gather*}
        \end{enumerate}
    \end{remark}
    
    The following result provides the variational formulation corresponding to the source condition \eqref{SC4}.

        \begin{theo}
        Assume that for almost any $t\in\Gamma$, assumptions \ref{a1}, \ref{a2}, \ref{a3} and \eqref{ineqbd} from Theorem \ref{generalcase} hold.
        Additionally, suppose that there exists $p\de\in Y^*$ single-valued and bounded almost everywhere, such that $A^*p\de\in\partial R(u\de)$ and a function $\Phi:[0,+\infty)\rightarrow[0,+\infty)$ strictly increasing, continuous and concave such that $\Phi(0)=0$ and 
        \begin{equation*}
            \forall \; p\in Y^*,\quad \left\langle \frac{p\de}{\varphi_f''(\cdot,f(\cdot))},p\de-p\right\rangle \leq \Phi(D_{R^*}(A^*p,A^*p\de)),
        \end{equation*}
        where the derivative is taken with respect to the second variable of $\varphi_f$.
        Then one gets the following estimate, $$D_R(u_{\alpha},u\de) \leq 2\Psi(\alpha)+\Phi^{-1}\left(\alpha^2\overline{C}\right),$$ where $\Psi$ is the conjugate of the convex mapping $t\mapsto\Phi^{-1}(t)$ and $\overline{C}$ is a constant.
    \end{theo}
    The proof follows similarly as in the Kullback-Leibler context.

    One can also adapt Theorem \ref{connection}  to find a connection between the variational assumption from the previous theorem and the range of $A^{**}$.
    \begin{theo}
        Assume that source condition \eqref{SC1} holds. Let  $\Phi:[0,+\infty)\rightarrow[0,+\infty)$ be a strictly increasing, continuous, concave function, such that $\Phi(0)=0$ and 
        \begin{equation*}
            \forall \; p\in Y^*,\quad \left\langle\frac{p\de}{\varphi_f''(\cdot,f(\cdot))},p\de-p\right\rangle \leq \Phi(D_{R^*}(A^*p,A^*p\de)),
        \end{equation*}
        where the derivative is taken with respect to the second variable of $\varphi_f$.
        If  $R^*$ is twice Fr\'{e}chet differentiable at $A^*p\de$ and $\Phi(t)\sim t^{\frac{1}{2}}$ as $t\rightarrow0$, then $\displaystyle \frac{p\de}{\varphi_f''(\cdot,f(\cdot))}\in\operatorname{Ran}(A^{**})$ holds.
    \end{theo}
\section*{Acknowledgements}
    The authors would like to express their gratitude to Barbara Kaltenbacher and Tobias Wolf from the University of Klagenfurt, Austria, for their suggestions and constructive feedback. The support  by the Austrian Science Fund (FWF) DOC 78 is acknowledged. The useful and constructive comments of the referees have helped to greatly improve the manuscript.

\printbibliography
\end{document}